\documentclass[reqno, a4paper, 11pt]{amsart}
\usepackage{amsfonts}
\usepackage[dvipsnames]{xcolor}
\usepackage{amssymb}
\usepackage{amsthm}
\usepackage{latexsym,amsmath}
\usepackage{subfig}
\usepackage{graphicx}
\usepackage{float}
\usepackage{afterpage}
\usepackage[T1]{fontenc}
\usepackage{epigraph}
\usepackage[shortlabels]{enumitem}
\usepackage{hyperref}
\usepackage{geometry}
\usepackage{bm}
\usepackage{multicol}
\usepackage{orcidlink}

\newcommand{\R}{\mathbb{R}}

\newcommand{\N}{\mathbb{N}}

\newcommand{\V}{\mathbb{V}}
\newcommand{\W}{\mathbb{W}}

\newcommand{\HB}{\mathbb{H}}
\newcommand{\LB}{\mathbb{L}}
\newcommand{\T}{\mathbb{T}}

\theoremstyle{plain}
\newtheorem{theorem}{Theorem}[section]
\newtheorem{lemma}[theorem]{Lemma}

\newtheorem{proposition}[theorem]{Proposition}

\theoremstyle{definition}
\newtheorem{definition}[theorem]{Definition}

\theoremstyle{remark}

\numberwithin{equation}{section}

\newcommand{\RR}{\mathcal{R}}
\newcommand{\lnorm}[2]{\left\|#1\right\|_{\LB^{#2}(\Omega)}}
\newcommand{\lnormt}[2]{\left\|#1\right\|_{\LB^{#2}(\Omega_T)}}
\newcommand{\hnorm}[2]{\left\|#1\right\|_{\HB^{#2}(\Omega)}}
\newcommand{\hnormt}[2]{\left\|#1\right\|_{\HB^{#2}(\Omega_T)}}
\newcommand{\wnorm}[3]{\left\|#1\right\|_{W^{#2,#3}(\Omega)}}

\newcommand{\liprod}[2]{\left\langle#1,#2\right\rangle_{\LB^2(\Omega)}}
\newcommand{\liprodt}[2]{\left\langle#1,#2\right\rangle_{\LB^2(\Omega_T)}}

\newcommand{\intomega}[1]{\int_{\Omega} #1 \; {\rm d} \vec{x}}

\newcommand{\intOt}[1]{\int_0^t #1 \; {\rm d}s}
\renewcommand{\vec}[1]{\mbox{\boldmath $ #1 $}}
\newcommand{\femvec}[3]{\vec{#1}_{#2}^{(#3)}}
\newcommand{\ddt}{\frac{\rm{d}}{{\rm d}t}}
\newcommand{\ds}{{\rm d}s}

\newcommand{\weakto}{\rightharpoonup}
\newcommand{\wkstarto}{\overset{\star}{\rightharpoonup}}
\newcommand{\lsim}{\lesssim}

\newcounter{Ax}
\newcommand{\itemA}{%
    \addtocounter{Ax}{1}
    \item[(A\theAx)]}

\allowdisplaybreaks

\newcommand{\tha}[1]{\textcolor{blue}{#1}}

\usepackage[normalem]{ulem}
\newcommand\tsout{\bgroup\markoverwith{\textcolor{red}{\rule[0.5ex]{2pt}{1.4pt}}}\ULon}
\newcommand{\stkout}[1]{\ifmmode\text{\tsout{\ensuremath{#1}}}\else\tsout{#1}\fi}

\title[Well-Posedness and FEM Approximation for the LLG Equation with Spin-Torques]{Well-Posedness and Finite Element Approximation for the Landau-Lifshitz-Gilbert Equation with Spin-Torques}
\author{{Noah Vinod\orcidlink{0009-0008-4388-2329} and Thanh Tran\orcidlink{0000-0001-6117-4811}}}
\date{15 July 2024}
\address{School of Mathematics and Statistics, The University of New South Wales, Sydney 2052, Australia}
\email{\tha{n.vinod@unsw.edu.au}}
\address{School of Mathematics and Statistics, The University of New South Wales, Sydney 2052, Australia}
\email{\tha{thanh.tran@unsw.edu.au}}

\begin{document}

\begin{abstract}
Spin currents act on ferromagnets by exerting a torque on the magnetisation. This torque is modelled by appending additional terms to the Landau-Lifshitz-Gilbert equation motivating the study of the non-homogeneous Landau-Lifshitz-Gilbert equation. We first prove the existence and uniqueness of high regularity local solutions to this equation using the Faedo-Galerkin method. Then we construct a numerical method for the problem and prove that it converges to a global weak solution of the PDE. Numerical simulations of the problem are also included.
\end{abstract}

\maketitle

\section{Introduction} \label{section-introduction}

In this paper, we develop local-in-time existence and uniqueness theorems as well as a global-in-time numerical method for the non-homogeneous Landau-Lifshitz-Gilbert initial boundary value problem
\begin{subequations} \label{llg-mod-pde}
	\begin{alignat}{2}
	& \frac{\partial \vec{m}}{\partial t} = \alpha \vec{m} \times \Delta
		\vec{m} - \beta \vec{m} \times \frac{\partial \vec{m}}{\partial
	t} + \vec{f}\left(\vec{m}, \nabla \vec{m}\right)
	\quad && \text{on } (0, T) \times \Omega, 
	\label{eq:llg-mod-pde a}
	\\[2ex] 
	& \frac{\partial \vec{m}}{\partial \vec{n}} = \vec{0} && \text{on } (0,T) \times \partial \Omega, 
	\label{eq:llg-mod-pde b}
	\\[2ex]
	& \vec{m}(0,\cdot) = \vec{m}_0 && \text{on } \Omega, 
	\label{eq:llg-mod-pde c}
	\\[2ex]
	& |\vec{m}| = 1 && \text{on } [0,T] \times \Omega,
	\label{eq:llg-mod-pde d}
    \end{alignat}
\end{subequations}
in $\R^d$ where $d=1,2,3$. Here $T > 0$, $\Omega$ is a bounded domain in $\R^d$ with smooth boundary $\partial \Omega$, $\vec{n}$ is the outward unit normal vector to $\partial\Omega$, $\vec{m} : [0,T] \times \Omega \to \R^3$ is the magnetisation vector and $\alpha,\beta$ are arbitrary positive constants. The function $\vec{f}$ is given and might depend on $t$ and $\vec{x}$ but we omit it here for the sake of simplicity. We will specify more assumptions on~$\vec{f}$ in a moment.

The Landau-Lifshitz-Gilbert equation (i.e., $\vec{f} = \vec{0}$) was developed by Landau and Lifshitz when studying the distribution of magnetic moments in a ferromagnetic crystal \cite{landau1935-article} and was improved upon by Gilbert \cite{gilbert1955-article}. It is very complex and non-integrable \cite{lakshmanan2011-article}. Its high non-linearity makes it very difficult to solve and thus interesting to mathematicians who seek to study the existence, uniqueness and regularity of solutions. Over the years, several studies for local/global existence, uniqueness/non-uniqueness, regularity, partial regularity and more have been conducted in $\R^d$ where $d=1,2,3$. A non-exhaustive list includes
\cite{alouges1992-article, carbou2001-article, chen2000-article, chen1998-article, harpes2004-article, melcher2005-article, visintin1985-article, zhou1981-article}. Likewise, there have been a number of studies conducted into numerical approximation methods for solutions to this equation related models, of which a non-exhaustive list includes \cite{alouges2008-article, bartels2008-article, bartels2006vol44-article, banas2008-article, le2013-article}.

In recent times, physicists have been conducting experiments to examine the effects that a current has on ferromagnets. In particular, an electric or spin current acting on a ferromagnet exerts a torque on the magnetisation \cite{abert2015-article}. The torque effect can be modelled by the Landau-Lifshitz-Gilbert equation by appending an additional torque term \cite{ado2017-article, garello2013-article} which is often a function of $\vec{m}$ and may involve its spatial derivatives \cite{li2004-article, meo2022-article, ralph2008-article}. In the past, mathematicians have considered these torques on a case-by-case basis with the Landau-Lifshitz-Gilbert equation and the Landau-Lifshitz-Bloch equation (cf. \cite{ayouch2021-article, melcher2013-article}). Here, we attempt to generalise the phenomena for the Landau-Lifshitz-Gilbert equation by appending a non-homogeneous term, $\vec{f}$, to the Landau-Lifshitz-Gilbert equation. This term represents the various torque effects that could be accounted for in the Landau-Lifshitz-Gilbert equation (see Section \ref{section-applications} for examples). Hence we are prompted to examine the conditions for the existence, uniqueness and regularity of solutions to the non-homogeneous Landau-Lifshitz-Gilbert (NHLLG) equation and also seek to develop a convergent numerical method for it. We extend the analysis in \cite{carbou2001-article} to prove the former while using Alouges' method \cite{alouges2008-article} to produce the latter. Theoretically, we prove the existence and uniqueness of local-in-time strong solutions for $d=1,2,3$ and numerically, we show that our method converges to a global-in-time weak solution for $d=2,3$.

Now returning to $\vec{f}$, we assume that $\vec{f}$ satisfies the following properties:
\begin{enumerate}
    \itemA $\vec{f} : \R^3 \times \R^{3 \times d} \to \R^3$ is a separable function such that
    \begin{equation*}
        \vec{f}(\vec{a}, \vec{B}) = \vec{f}_1(\vec{a}) + \vec{f}_2(\vec{a}) \times \vec{g}_1(\vec{B}) + \vec{f}_3(\vec{a}) \times (\vec{f}_4(\vec{a}) \times \vec{g}_2(\vec{B})),
    \end{equation*}
    where $\vec{f}_1, \vec{f}_2, \vec{f}_3: \R^3 \to \R^3$ and $\vec{g}_1, \vec{g}_2 : \R^{3 \times d} \to \R^3$ are $C^1$ functions satisfying $\vec{f}_i(\vec{0})= \vec{g}_i(\vec{0}) = \vec{0}$ for all possible $i$ with the additional property that $\vec{g}_1$ and~$\vec{g}_2$ are linear.
    
    \itemA For all $(\vec{a}, \vec{B}) \in \R^3 \times \R^{3 \times d}$ the function~$\vec{f}$ satisfies
    \begin{equation*}
        \vec{f}(\vec{a},\vec{B}) \cdot \vec{a} = 0,
    \end{equation*}
    and the component functions satisfy the growth conditions
    \begin{align*}
        |\vec{f}_i(\vec{a})| &\lsim 1 + |\vec{a}|^p, \\
        |\nabla \vec{f}_i(\vec{a})| &\lsim 1 + |\vec{a}|^p,
    \end{align*}
    where $i =1,2,3,4$ and $p$ is some positive constant.
\end{enumerate}
Justification for these assumptions will be presented in Section \ref{section-applications}.

The paper is organised as follows. We lay down in Section \ref{section-preliminaries} the notation and some important results. Section \ref{section-proof-1} is devoted to the statement and proofs of the local-in-time existence and uniqueness results for the high regularity local solutions. Section
\ref{sec:numerical-method} is devoted to the description and proof of convergence of the numerical method as it tends to a global-in-time lower regularity solution. Finally, in Section \ref{section-applications} we give examples of such non-homogeneous terms from the literature that satisfy our assumptions and run some numerical experiments for them.

\section{Preliminaries} \label{section-preliminaries}

In this paper $C$ denotes a generic positive constant which may take different values at different occurrences. We also use the notation $a\lesssim b$ in lieu of $a \leq Cb$ when it is not necessary to clarify the constant~$C$.

\begin{lemma} \label{carbou-lemma}
Let $\Omega$ be a bounded domain of  $\R^d$, $d=1,2,3$, with smooth boundary. Then for all $\vec{u} \in \HB^2(\Omega)$ such that $\displaystyle \frac{\partial \vec{u}}{\partial\vec{n}} = \vec{0}$ on $\partial \Omega$,
\begin{equation} \label{carbou-lemma:eqn 1}
    \hnorm{\vec{u}}{2} \lsim \left(\lnorm{\vec{u}}{2}^2 + \lnorm{\Delta \vec{u}}{2}^2\right)^{1/2}.
\end{equation}
Furthermore, for $\vec{u} \in \HB^3(\Omega)$ such that $\displaystyle \frac{\partial \vec{u}}{\partial \vec{n}} = \vec{0}$ on $\partial\Omega$,
\begin{align}
    \hnorm{\nabla \vec{u}}{2} &\lsim \left(\lnorm{\nabla \vec{u}}{2}^2 + \lnorm{\Delta \vec{u}}{2}^2 + \lnorm{\nabla \Delta \vec{u}}{2}^2\right)^{1/2}, \label{carbou-lemma:eqn 2} \\
    \lnorm{\nabla^2 \vec{u}}{3} &\lsim \hnorm{\vec{u}}{2} + \hnorm{\vec{u}}{2}^{1/2} \lnorm{\nabla \Delta \vec{u}}{2}^{1/2}. \label{carbou-lemma:eqn 3}
\end{align}
\end{lemma}
\begin{proof}
Inequality \eqref{carbou-lemma:eqn 1} results from the regularity of solutions of $-\Delta \vec{u} + \vec{u} = \vec{g}$ (for some $\vec{g} \in \LB^2(\Omega)$) subject to the homogeneous Neumann boundary condition. This boundary condition also enables us to prove Inequality \eqref{carbou-lemma:eqn 2} using Proposition~1.4 of Appendix~I in \cite{temam2001-book} with $\vec{v} = \nabla \vec{u}$. The last inequality (\ref{carbou-lemma:eqn 3}) issues from the Gagliardo-Nirenberg inequality
\begin{equation*}
    \wnorm{\vec{u}}{2}{3} \lsim \hnorm{\vec{u}}{2}^{1/2} \hnorm{\vec{u}}{3}^{1/2}
\end{equation*}
and \eqref{carbou-lemma:eqn 2}.
\end{proof}

Using the elementary identities
\begin{align*}
    \vec{u} \times (\vec{v} \times \vec{w}) &= (\vec{u} \cdot \vec{w}) \vec{v} - (\vec{u} \cdot \vec{v}) \vec{w}, \\
    \Delta (|\vec{u}|^2) &= 2|\nabla \vec{u}|^2 + 2(\vec{u} \cdot \Delta \vec{u}),
\end{align*}
and the fact that $|\vec{m}|=1$, one can prove that equation~\eqref{eq:llg-mod-pde a} is
equivalent to
\begin{equation} \label{eq:llg-numerical-pde}
		\beta \frac{\partial \vec{m}}{\partial t} - \vec{m} \times \frac{\partial \vec{m}}{\partial t} 
		= 
		\alpha|\nabla \vec{m}|^2 \vec{m} + \alpha \Delta \vec{m} 
		- \vec{m} \times \vec{f}\left(\vec{m},\nabla \vec{m}\right)
\end{equation}
and
\begin{equation} \label{eq:llg-theoretical-pde}
    \frac{\partial \vec{m}}{\partial t} - \beta'\Delta \vec{m} 
    = 
    \beta'|\nabla \vec{m}|^2 \vec{m} + \alpha'\vec{m} \times \Delta \vec{m} 
    + \vec{F}\left(\vec{m}, \nabla \vec{m}\right),
\end{equation}
where
\begin{equation*}
    \alpha' = \frac{\alpha}{1+\beta^2}, \quad
	\beta' = \frac{\alpha\beta}{1+\beta^2},
\end{equation*}
and
\begin{equation}\label{equ:F}
	\vec{F}\left(\vec{m}, \nabla \vec{m}\right) 
	= -\frac{\beta}{1 + \beta^2} \vec{m} 
	\times \vec{f}\left(\vec{m}, \nabla \vec{m}\right) 
	+ \frac{1}{1 + \beta^2} \vec{f}\left(\vec{m}, \nabla \vec{m}\right).	
\end{equation}
It follows from assumptions (A1) and (A2) that~$\vec{F}$ has the following properties:

\begin{lemma} \label{lem:locally-lipschitz}
Let $\vec{u}, \vec{v} \in B_R(\vec{0})$, where $B_R(\vec{0})$ is the open ball with centre zero and radius $R$ in~$\HB^2(\Omega)$. Then 
\begin{equation*}
    \lnorm{\vec{F}\left(\vec{u}, \nabla \vec{u}\right) - \vec{F}\left(\vec{v}, \nabla \vec{v}\right)}{2} \lsim \hnorm{\vec{u} - \vec{v}}{1}
\end{equation*}
where the constant can depend on $R$.
\end{lemma}
\begin{proof}
We first show that $\vec{f}$ satisfies these properties from which it is easy to see why $\vec{F}$ would too, with possibly different constants. Note $\vec{f}_1$ is locally Lipschitz because it is continuously differentiable and because of the Sobolev embedding $\HB^2(\Omega) \hookrightarrow \LB^{\infty}$. Therefore, it follows that
\begin{equation*}
    \lnorm{\vec{f}_1(\vec{u}) - \vec{f}_1(\vec{v})}{2} \lsim \lnorm{\vec{u} - \vec{v}}{2}.
\end{equation*}
Similarly, by noting that $\vec{f}_2$ is bounded on compact sets, because it is $C^1$, that $\vec{g}_1$ is linear and $\HB^1(\Omega)$ embeds into $\LB^4(\Omega)$ we obtain
\begin{align*}
    &\lnorm{\vec{f}_2(\vec{u}) \times \vec{g}_1\left(\nabla \vec{u}\right) - \vec{f}_2(\vec{v}) \times \vec{g}_1\left(\nabla \vec{v}\right)}{2} \\
    &\qquad \lsim \lnorm{\vec{f}_2(\vec{u}) \times \left(\vec{g}_1\left(\nabla \vec{u}\right) - \vec{g}_1\left(\nabla \vec{v}\right)\right)}{2} + \lnorm{(\vec{f}_2(\vec{u}) - \vec{f}_2(\vec{v})) \times \vec{g}_1\left(\nabla \vec{v}\right)}{2} \\
    &\qquad \lsim \lnorm{\vec{f}_2(\vec{u})}{\infty} \lnorm{\vec{g}_1\left(\nabla \vec{u}\right) - \vec{g}_1\left(\nabla \vec{v}\right)}{2} + \lnorm{\vec{f}_2(\vec{u}) - \vec{f}_2(\vec{v})}{4} \lnorm{\vec{g}_1\left(\nabla \vec{v}\right)}{4} \\
    &\qquad \lsim \lnorm{\nabla \vec{u} - \nabla \vec{v}}{2} + \hnorm{\vec{u} - \vec{v}}{1} \hnorm{\nabla \vec{v}}{1} \\
    &\qquad \lsim \hnorm{\vec{u} - \vec{v}}{1}.
\end{align*}
Doing something similar, same can be shown for the third term in the definition of $\vec{f}$. Therefore, we conclude that $\vec{f}$ satisfies the first property and so does $\vec{F}$.
\end{proof}

\begin{lemma} \label{lem:growth-condition}
For any $\vec{u} \in \HB^2(\Omega)$,
\begin{equation*}
    \hnorm{\vec{F}\left(\vec{u}, \nabla \vec{u}\right)}{1} \lsim 1 + \hnorm{\vec{u}}{2}^{\mu},
\end{equation*}
where~$\mu=2p+2$ with~$p$ defined in Assumption~(A2).
\end{lemma}
\begin{proof}
First we show that~$\vec{f}_i(\vec{u})$ and~$\vec{g}_j(\nabla\vec{u})$ belong to~$\HB^1(\Omega)$ for all~$i=1,\ldots,4$ and~$j=1,2$. This can be easily seen from \cite[Proposition 9.5]{brezis2010-book}, the growth conditions in assumption (A2), and the linearity of $\vec{g}_j$. More precisely, $\vec{f}_i(\vec{u})$ and $\vec{g}_j(\nabla \vec{u})$ satisfy
\begin{align*}
    |\nabla \vec{f}_i(\vec{u})| &\lsim |\nabla_{\vec{a}} \; \vec{f}_i(\vec{u})| |\nabla \vec{u}| \lsim (1 + |\vec{u}|^p) |\nabla \vec{u}|, \\
    |\nabla \vec{g}_j(\nabla \vec{u})| &\lsim |\nabla_{\vec{B}} \; \vec{g}_j(\nabla \vec{u})| |\nabla^2 \vec{u}| \lsim |\nabla^2 \vec{u}|,
\end{align*}
for $i = 1,2,3,4$ and $j =1,2$. This implies $\vec{f}(\vec{u}, \nabla \vec{u})$ belongs to $\HB^1(\Omega)$. Next we show that $\vec{f}$ satisfies the required results. From the definition of $\vec{f}$ and assumption (A2) we have
\begin{align*}
    \hnorm{\vec{f}(\vec{u}, \nabla \vec{u})}{1} &\leq \hnorm{\vec{f}_1(\vec{u})}{1} + \hnorm{\vec{f}_2(\vec{u}) \times \vec{g}_1(\nabla \vec{u})}{1} + \hnorm{\vec{f}_3(\vec{u}) \times (\vec{f}_4(\vec{u}) \times \vec{g}_2(\nabla \vec{u}))}{1} \\
    &\lsim 1 + \hnorm{\vec{u}}{2}^p + \hnorm{\vec{f}_2(\vec{u}) \times \vec{g}_1(\nabla \vec{u})}{1} + \hnorm{\vec{f}_3(\vec{u}) \times (\vec{f}_4(\vec{u}) \times \vec{g}_2(\nabla \vec{u}))}{1}
\end{align*}
We now obtain bounds for the third term. We have
\begin{align*}
    |\vec{f}_2(\vec{u}) \times \vec{g}_1(\nabla \vec{u})| &\lsim (1 + |\vec{u}|^p) |\nabla \vec{u}|
\end{align*}
and
\begin{align*}
    |\nabla \left(\vec{f}_2(\vec{u}) \times \vec{g}_1(\nabla \vec{u})\right)| &\lsim |\nabla \vec{f}_2(\vec{u})| |\vec{g}_1(\nabla \vec{u})| + |\vec{f}_2(\vec{u})| |\nabla \vec{g}_2(\nabla \vec{u})| \\
    &\lsim (1 + |\vec{u}|^p) |\nabla \vec{u}|^2 + (1 + |\vec{u}|^p) |\nabla^2 \vec{u}|,
\end{align*}
so that
\begin{align*}
    \hnorm{\vec{f}_2(\vec{u}) \times \vec{g}_1(\nabla \vec{u})}{1} &\lsim \left(1 + \lnorm{\vec{u}}{\infty}^p \right) \left(\lnorm{\nabla \vec{u}}{2} +\lnorm{\nabla \vec{u}}{4}^2 + \lnorm{\nabla^2 \vec{u}}{2}\right) \\[1ex]
    &\lsim 1 + \hnorm{\vec{u}}{2}^{p+2}.
\end{align*}
Similarly, for the last term we have,
\begin{equation*}
    \hnorm{\vec{f}_3(\vec{u}) \times (\vec{f}_4(\vec{u}) \times \vec{g}_2(\nabla \vec{u}))}{1} \lsim 1 + \hnorm{\vec{u}}{2}^{2p+2}.
\end{equation*}
The required result for $\vec{f}$ holds with $\mu=2p+2$, which means we have $\mu = 2p+3$ for $\vec{F}$.
\end{proof}

Due to the equivalence of \eqref{eq:llg-mod-pde a} and \eqref{eq:llg-theoretical-pde}, problem \ref{llg-mod-pde} is equivalent to the following problem:
\begin{subequations} \label{eq:llg-equiv-pde}
	\begin{alignat}{2}
		&\frac{\partial \vec{m}}{\partial t} - \beta'\Delta \vec{m} 
		= 
		\beta'|\nabla \vec{m}|^2 \vec{m} + \alpha'\vec{m} \times \Delta \vec{m} 
		+ \vec{F}\left(\vec{m},\nabla \vec{m}\right)
		\quad && \text{on } (0, T) \times \Omega, 
		\label{eq:llg-equiv-pde a}
		\\[1ex]
	& \frac{\partial \vec{m}}{\partial \vec{n}} = \vec{0} && \text{on } (0,T) \times \partial \Omega, 
	\label{eq:llg-equiv-pde b}
	\\[1ex]
	& \vec{m}(0,\cdot) = \vec{m}_0 && \text{on } \Omega, 
	\label{eq:llg-equiv-pde c}
	\\[1ex]
	& |\vec{m}| = 1 && \text{on } [0,T] \times \Omega,
	\label{eq:llg-equiv-pde d}
    \end{alignat}
\end{subequations}
We focus on~\eqref{eq:llg-equiv-pde} for the proofs of the existence and uniqueness of high regularity local solutions. We now state the definition of the solution.

\begin{definition}[Existence and Uniqueness] \label{def-weak-solution-theoretical}
Given $T > 0$, a function $\vec{m} : [0,T] \times \Omega \to \R^3$ belonging to $L^2(0,T; \HB^2(\Omega)) \cap H^1(0,T; \LB^2(\Omega))$ is said to be a solution to \eqref{eq:llg-equiv-pde} if it satisfies 
\begin{enumerate}[(i)]
    \item $|\vec{m}(t,\vec{x})| \equiv 1$ for all $t \in [0,T]$ and almost all $\vec{x} \in \Omega$;
    \item for every $t \in [0,T]$ and for all $\boldsymbol{\psi} \in C^{\infty}(\Omega)$,
    \begin{align*} 
    	\liprod{\vec{m}(t)}{\boldsymbol{\psi}}
	&+
	\beta' \intOt{\liprod{\nabla\vec{m}(s)}{\nabla\boldsymbol{\psi}}}
	\\
	&=
	\beta' \intOt{\liprod{|\nabla\vec{m}(s)|^2 \vec{m}(s)}{\boldsymbol{\psi}}}
	-
	\alpha' \intOt{\liprod{\vec{m}(s) \times \nabla \vec{m}(s)}{\nabla \boldsymbol{\psi}}}
	\\
        &\quad 
	+ \intOt{\liprod{\vec{F}\left(\vec{m}(s), \nabla \vec{m}(s)\right)}{\boldsymbol{\psi}}}.
    \end{align*}
\end{enumerate}
Note that in this case, $\vec{m}$ satisfies \eqref{eq:llg-equiv-pde a}, and thus~\eqref{eq:llg-mod-pde a}, for all $t \in [0,T]$ and almost all $\vec{x} \in \Omega$.
\end{definition}


\section{Existence and Uniqueness} \label{section-proof-1}

The following existence theorem is primary result of this section.

\begin{theorem}[Existence] \label{exist-thm}
Assume that $\vec{f}$ satisfies the assumptions (A1) and (A2) and that the initial data $\vec{m}_0$ satisfies
\begin{equation} \label{initial-data-assumptions}
    \begin{cases}
        \vec{m}_0 \in \HB^2(\Omega),  \\[1ex]
        \displaystyle\frac{\partial \vec{m}_0}{\partial \vec{n}} = \vec{0} \text{ on } \partial \Omega, \\[1ex]
        |\vec{m}_0| = 1 \text{ on } \Omega.
    \end{cases}
\end{equation}
Given~$T>0$ and $d=1,2,3$, there exists $T^\ast >0$ satisfying $T^\ast \leq T$ and
\begin{equation*}
    \vec{m} \in C([0,T^\ast]; \HB^2(\Omega)) \cap\LB^2(0,T^\ast; \HB^3(\Omega))
\end{equation*}
such that $\vec{m}$ is a strong solution to \eqref{llg-mod-pde} in the sense of Definition \ref{def-weak-solution-theoretical}.
\end{theorem}

\begin{theorem}[Stability] \label{unique-thm}
Let $\vec{m}_1$ and $\vec{m}_2$ be two solutions given by Theorem \ref{exist-thm}, which are associated with the initial data $\vec{m}_{0,1}$ and $\vec{m}_{0,2}$ satisfying \eqref{initial-data-assumptions}, respectively. Then there exists a constant $C$ independent of $\vec{m}_{0,1}$ and $\vec{m}_{0,2}$ such that
\begin{equation*}
    \sup_{t \in [0,T']} \lnorm{\vec{m}_1(t,\cdot) - \vec{m}_2(t,\cdot)}{2} \leq C \lnorm{\vec{m}_{0,1} - \vec{m}_{0,2}}{2},
\end{equation*}
where $T' = \min(T_1^\ast, T_2^\ast)$. Consequently, if $\vec{m}_{0,1} = \vec{m}_{0,2}$ the solutions are unique.
\end{theorem}

\subsection{Faedo-Galerkin Method}

Let $\{\boldsymbol\varphi_j\}_{j=1}^{\infty}$ denote the orthonormal basis of $\LB^2(\Omega)$ where each $\boldsymbol\varphi_j$ is an eigenfunction of $-\Delta$ with the homogeneous Neumann boundary condition. That is,
$\boldsymbol{\varphi}_j$ satisfies
\begin{equation*}
    -\Delta \boldsymbol{\varphi}_j = \lambda_j \boldsymbol{\varphi}_j \quad \text{and} \quad \frac{\partial \boldsymbol{\varphi}_j}{\partial \vec{n}} = \vec{0} \text{ on } \partial \Omega.
\end{equation*}
Define $V_k := \mathrm{span}\{\boldsymbol\varphi_1, \ldots, \boldsymbol\varphi_k\}$ and $\Pi_k : \LB^2(\Omega) \to V_k$ to be the orthogonal projection with respect to the $\LB^2(\Omega)$ inner product. 

We find a Galerkin solution to \eqref{eq:llg-equiv-pde a} by solving for $\vec{m}_k \in V_k$ that satisfies
\begin{equation} \label{eq:proj-galerkin}
\begin{aligned} 
    \frac{\partial \vec{m}_k}{\partial t} - \beta'\Delta \vec{m}_k &= 
		\beta' \Pi_k (|\nabla \vec{m}_k|^2 \vec{m}_k) + \alpha' \Pi_k(\vec{m}_k \times \Delta \vec{m}_k) + \Pi_k(\vec{F}\left(\vec{m}_k, \nabla \vec{m}_k\right)), \\
  \vec{m}_k(0,\cdot) &= \Pi_k (\vec{m}_0).
\end{aligned}
\end{equation}

\begin{lemma} \label{lipschitz-lemma}
For each $k \in \N$ and $\vec{u} \in V_k$, define
\begin{align*}
    F_k^1(\vec{u}) &= \Delta \vec{u} \\
    F_k^2(\vec{u}) &= \Pi_k\left(|\nabla \vec{u}|^2 \vec{u}\right) \\
    F_k^3(\vec{u}) &= \Pi_k(\vec{u} \times \Delta \vec{u}) \\
    F_k^4(\vec{u}) &= \Pi_k\left(\vec{F}\left(\vec{u}, \nabla \vec{u}\right)\right).
\end{align*}
Each $F_k^i$ is a well-defined mapping from $V_k$ to $V_k$. For $i=1,4$ the functions are globally Lipschitz, while for $i=2,3$ the functions are locally Lipschitz.
\end{lemma}
\begin{proof}
The results for $F_k^1$ and $F_k^3$ are well known (see, e.g., \cite{soenjaya2023-article}), so we prove only the result for~$F_k^2$ and~$F_k^4$. Let $\vec{u} \in V_k$. As the eigenfunctions are smooth, so is $|\nabla \vec{u}|^2\vec{u}$. Therefore $F_k^2$ is well defined and we have
\begin{align*}
    \lnorm{F_k^2(\vec{u}) - F_k^2(\vec{u})}{2}^2 &\leq \lnorm{|\nabla \vec{u}|^2 \vec{u} - |\nabla \vec{v}|^2 \vec{v}}{2}^2 \\
    &\leq \lnorm{|\nabla \vec{u}|^2 (\vec{u} - \vec{v})}{2}^2 + \lnorm{(|\nabla \vec{u}|^2 - |\nabla \vec{v}|^2)\vec{v}}{2}^2 \\
    &\leq \lnorm{|\nabla \vec{u}|^2 (\vec{u} - \vec{v})}{2}^2 + \lnorm{(|\nabla \vec{u}| + |\nabla \vec{v}|)(|\nabla \vec{u} - \nabla \vec{v}|)\vec{v}}{2}^2 \\
    &\leq \lnorm{\nabla \vec{u}}{\infty}^4 \lnorm{\vec{u} - \vec{v}}{2}^2 \\
    &\qquad + \left(\lnorm{\nabla \vec{u}}{\infty}^2 + \lnorm{\nabla \vec{v}}{\infty}^2\right)\lnorm{\vec{v}}{\infty}^2\lnorm{\nabla \vec{u} - \nabla \vec{v}}{2}^2 \\[1ex]
    &\lsim \left(\lnorm{\nabla \vec{u}}{\infty}^4 +\left(\lnorm{\nabla \vec{u}}{\infty}^2 + \lnorm{\nabla \vec{v}}{\infty}^2\right)\lnorm{\vec{v}}{\infty}^2\right) \lnorm{\vec{u} - \vec{v}}{2}^2,
\end{align*}
since the $\HB^1$ norm is equivalent to the $\LB^2$ norm in the finite-dimensional space $V_k$. Similarly, $F_k^4$ is well defined and by Lemma \ref{lem:locally-lipschitz}
\begin{align*}
    \lnorm{F_k^4(\vec{u}) - F_k^4(\vec{v})}{2} \leq \lnorm{\vec{F}\left(\vec{u}, \nabla\vec{u}\right) - \vec{F}(\vec{v}, \nabla \vec{v})}{2} \lsim \hnorm{\vec{u} - \vec{v}}{1} \simeq \lnorm{\vec{u} - \vec{v}}{2},
\end{align*}
completing the proof of the lemma.
\end{proof}

The Cauchy-Lipschitz Theorem and Lemma \ref{lipschitz-lemma} guarantee that such a solution exists. As each eigenfunction belongs to $C^{\infty}(\overline{\Omega})$, so does $\vec{m}_k$. Now we prove some \textit{a priori} results for the Galerkin solution.

Problem \eqref{eq:proj-galerkin} is equivalent to the following problem. In other words, we have some $\vec{m}_k \in V_k$ such that
\begin{equation} \label{galerkin-eqn}
    \begin{aligned}
        \liprod{\partial_t \vec{m}_k}{\boldsymbol\varphi} + \beta'\liprod{\nabla \vec{m}_k}{\nabla \boldsymbol\varphi} &= \beta'\liprod{|\nabla \vec{m}_k|^2 \vec{m}_k}{\boldsymbol\varphi} \\
        &\qquad + \alpha'\liprod{\vec{m}_k \times \Delta \vec{m}_k}{ \boldsymbol\varphi} \\
        &\qquad + \liprod{\vec{F}(\vec{m}_k,\nabla\vec{m}_k)}{\boldsymbol\varphi}, \\
        \liprod{\vec{m}_k(0,\cdot)}{\boldsymbol{\varphi}} &= \liprod{\vec{m}_0}{\boldsymbol{\varphi}},
    \end{aligned}
\end{equation}
for all $\boldsymbol{\varphi} \in V_k$.

\subsection{\textit{A priori} Estimates}

\begin{proposition} \label{prop:linf-h2}
There exists a $T^\ast > 0$ such that $T^\ast \leq T$ and such that for each $k \in \N$ and $t \in [0,T^\ast)$
\begin{equation*}
    \sup_{0 \leq s \leq t} \hnorm{\vec{m}_k(s)}{2}^2 + \int_0^t \hnorm{\vec{m}_k(s,\cdot)}{3} \; \ds
    \lsim 1,
\end{equation*}
where the constant is independent of $k$ and $T^\ast$.
\end{proposition}
\begin{proof}
Setting $\boldsymbol{\varphi} = \vec{m}_k$ in (\ref{galerkin-eqn}) and using assumption (A2) of $\vec{f}$, which is also satisfied by $\vec{F}$, we have
\begin{align*}
     \frac{1}{2} \ddt \lnorm{\vec{m}_k}{2}^2 + \beta'\lnorm{\nabla \vec{m}_k}{2}^2 &\lsim \lnorm{\vec{m}_k}{\infty}^2 \lnorm{\nabla \vec{m}_k}{2}^2.
\end{align*}
Then by Lemma \ref{carbou-lemma} and the fact that $\lnorm{\vec{m}_k}{\infty} \lsim \hnorm{\vec{m}_k}{2}$ we obtain
\begin{equation} \label{mk-estimate}
    \ddt \lnorm{\vec{m}_k}{2}^2  + 2 \beta'\lnorm{\nabla \vec{m}_k}{2}^2 \lsim \left(\|\vec{m}_k\|_{\LB^2(\Omega)}^2 + \|\Delta \vec{m}_k\|_{\LB^2(\Omega)}^2\right)^2.
\end{equation}
Setting $\boldsymbol{\varphi} = \Delta^2 \vec{m}_k$ in (\ref{galerkin-eqn}) we obtain
\begin{equation} \label{delta-mk-estimate}
    \ddt \lnorm{\Delta \vec{m}_k}{2}^2 + 2\beta'\lnorm{\nabla \Delta \vec{m}_k}{2}^2 \lsim |I_1| + |I_2| + |I_3|,
\end{equation}
where
\begin{align*}
    I_1 &:= -\intomega{\nabla (|\nabla \vec{m}_k|^2 \vec{m}_k) \cdot \nabla \Delta \vec{m}_k}, \\
    I_2 &:= -\intomega{(\nabla\vec{m}_k \times \Delta \vec{m}_k) \cdot \nabla \Delta \vec{m}_k}, \\
    I_3 &:= -\intomega{\nabla \vec{F}(\vec{m}_k,\nabla \vec{m}_k) \cdot \nabla \Delta \vec{m}_k}.
\end{align*}
By adding (\ref{mk-estimate}) and (\ref{delta-mk-estimate}) together, we arrive at the following inequality 
\begin{align} \label{l2-h2-estimate-sum}
    \begin{split}
         \ddt \left(\|\vec{m}_k\|_{\LB^2(\Omega)}^2 + \|\Delta \vec{m}_k\|_{\LB^2(\Omega)}^2\right) &+ 2\beta'\|\nabla \vec{m}_k\|_{\LB^2(\Omega)}^2 + 2\beta'\|\nabla \Delta \vec{m}_k\|_{\LB^2(\Omega)}^2 \\
        &\lsim \left(\|\vec{m}_k\|_{\LB^2(\Omega)}^2 + \|\Delta \vec{m}_k\|_{\LB^2(\Omega)}^2\right)^2 + |I_1| + |I_2| + |I_3|.
    \end{split}
\end{align}
Expanding the terms in $I_1$ we obtain by using Holder's inequality,
\begin{align*}
    |I_1| &\lsim \int_{\Omega} |\nabla \vec{m}_k|^3 |\nabla \Delta \vec{m}_k| \; d \vec{x} + \int_{\Omega} |\nabla^2 \vec{m}_k| |\nabla \vec{m}_k| |\vec{m}_k| |\nabla \Delta \vec{m}_k| \; d \vec{x}, \\
    &\leq \|\nabla \vec{m}_k\|_{\LB^6(\Omega)}^3 \|\nabla \Delta \vec{m}_k\|_{\LB^2(\Omega)} + \|\vec{m}_k\|_{\LB^{\infty}(\Omega)} \|\nabla^2 \vec{m}_k\|_{\LB^3(\Omega)} \|\nabla \vec{m}_k\|_{\LB^6(\Omega)} \|\nabla \Delta \vec{m}_k\|_{\LB^2(\Omega)}.
\end{align*}
We can similarly bound $I_2$ as
\begin{align*}
    |I_2| &\lsim \|\nabla \vec{m}_k\|_{\LB^6(\Omega)} \|\Delta \vec{m}_k\|_{\LB^3(\Omega)} \|\nabla \Delta \vec{m}_k\|_{\LB^2(\Omega)}.
\end{align*}
For $I_3$ we make use of Lemmas \ref{lem:growth-condition} and \ref{carbou-lemma} to obtain
\begin{align*}
    |I_3| &\lsim \lnorm{\nabla \vec{F}(\vec{m}_k, \nabla \vec{m}_k)}{2} \lnorm{\nabla \Delta \vec{m}_k}{2} \\
    &\lsim \left(1 + \hnorm{\vec{m}_k}{2}^{\mu}\right) \lnorm{\nabla \Delta \vec{m}_k}{2} \\
    &\lsim \left(1 + \left(\|\vec{m}_k\|_{\LB^2(\Omega)}^2 + \|\Delta \vec{m}_k\|_{\LB^2(\Omega)}^2\right)^{\mu/2}\right) \|\nabla \Delta \vec{m}_k\|_{\LB^2(\Omega)}.
\end{align*}
Let
\begin{align*}
    A_k(t) &= \lnorm{\vec{m}_k(t)}{2}^2 + \lnorm{\Delta \vec{m}_k(t)}{2}^2, \\
    B_k(t) &= \lnorm{\nabla \Delta \vec{m}_k(t)}{2}^2,
\end{align*}
Applying the fact that $\lnorm{\Delta \vec{m}_k}{3} \lsim \lnorm{\nabla^2 \vec{m}_k}{3}$ and Lemma \ref{carbou-lemma}, we write these estimates in terms of $A_k(t)$ and $B_k(t)$. That is,
\begin{align*}
    |I_1| + |I_2| + |I_3| &\lsim A_k^{3/2}(t) B_k^{1/2}(t) + A_k^{5/4}(t) B_k^{3/4}(t) + A_k(t) B_k^{1/2}(t) + A_k^{3/4}(t) B_k^{3/4}(t) \\
    &\qquad+ \left(1 + A_k^{\mu/2}(t) \right)B_k^{1/2}(t).
\end{align*}
Hence for sufficiently small $\varepsilon$, we conclude from \eqref{l2-h2-estimate-sum} that
\begin{equation*}
    A_k(t) + \int_0^t B_k(s) \; \ds \lsim A_k(0) + \int_0^t 1 + A_k^2(s) + A_k^3(s) + A_k^5(s) + A_k^{ \mu}(s) \; \ds.
\end{equation*}
The Generalised Gronwall Inequality (Theorem \ref{generalised-gronwall-lemma}) guarantees that there exists some $T^* \in (0,T]$ such that for any $s \in [0,T^*]$ we have
\begin{equation*}
    \sup_{s \in [0,t]} A_k(s) \lsim 1.
\end{equation*}
The result follows from Lemma \ref{carbou-lemma}. 

For the second part, we use (\ref{l2-h2-estimate-sum}) and integrate over $[0,t]$ to obtain
\begin{equation*}
    \int_0^t \lnorm{\nabla \vec{m}_k(s,\cdot)}{2}^2 + \lnorm{\nabla \Delta \vec{m}_k(s,\cdot)}{2}^2 \; \ds \lsim 1.
\end{equation*}
Using Lemma \ref{carbou-lemma} we obtain
\begin{align*}
    \int_0^t \hnorm{\vec{m}_k(s,\cdot)}{3}^2 \; \ds &= \int_0^t \lnorm{\vec{m}_k(s,\cdot)}{2}^2 + \hnorm{\nabla \vec{m}_k(s,\cdot)}{2}^2 \; \ds \\[1ex]
    &\lsim \int_0^t \lnorm{\vec{m}_k(s,\cdot)}{2}^2 + \lnorm{\nabla \vec{m}_k(s,\cdot)}{2}^2 \; \ds \\
    &\qquad + \int_0^t \lnorm{\Delta \vec{m}_k(s,\cdot)}{2}^2 + \lnorm{\nabla \Delta \vec{m}_k(s,\cdot)}{2}^2 \; \ds,
\end{align*}
which is bounded by the first part, completing the proof.
\end{proof}

The next two propositions deal with estimates on the time derivative of the Galerkin solution i.e., $\partial_t \vec{m}_k$.

\begin{proposition} \label{t-mk-proposition}
Let $T^*$ be as in Proposition \ref{prop:linf-h2}. For each $k \in \N$ and $t \in [0,T^*]$
\begin{equation*}
    \sup_{0 \leq s \leq t} \lnorm{\partial_t \vec{m}_k(s)}{2}^2 \lsim 1.
\end{equation*}
\end{proposition}
\begin{proof}
Setting $\boldsymbol{\varphi} = \partial_t \vec{m}_k$ in (\ref{galerkin-eqn}) we obtain
\begin{align*}
    \lnorm{\partial_t \vec{m}_k}{2}^2 &\lsim \left|\liprod{\Delta \vec{m}_k}{\partial_t \vec{m}_k}\right| + \left|\liprod{|\nabla \vec{m}_k|^2 \vec{m}_k}{\partial_t \vec{m}_k}\right| \\
    &\qquad + \left|\liprod{\vec{m}_k \times \Delta \vec{m}_k}{\partial_t \vec{m}_k}\right| + \left|\liprod{\vec{F}(\vec{m}_k,\nabla \vec{m}_k)}{\partial_t \vec{m}_k}\right| \\[1ex]
    &\lsim \|\partial_t \vec{m}_k\|_{\LB^2(\Omega)} \Big(\|\Delta \vec{m}_k\|_{\LB^2(\Omega)} + \||\nabla \vec{m}_k|^2\vec{m}_k\|_{\LB^2(\Omega)} \\
    &\qquad + \|\vec{m}_k \times \Delta\vec{m}_k\|_{\LB^2(\Omega)} + \|\vec{F}(\vec{m}_k, \nabla \vec{m}_k)\|_{\LB^2(\Omega)}\Big).
\end{align*}
Thus,
\begin{align*}
    \lnorm{\partial_t \vec{m}_k}{2} &\lsim \|\Delta \vec{m}_k\|_{\LB^2(\Omega)} + \|\vec{m}_k\|_{\LB^{\infty}(\Omega)} \|\nabla \vec{m}_k\|_{\LB^4(\Omega)}^2 \\
    &\qquad + \|\vec{m}_k\|_{\LB^{\infty}(\Omega)} \|\Delta \vec{m}_k\|_{\LB^2(\Omega)} \\
    &\qquad + \|\vec{F}(\vec{m}_k, \nabla \vec{m}_k)\|_{\LB^2(\Omega)}.
\end{align*}
The result follows from the Sobolev embedding $\HB^1(\Omega) \hookrightarrow \LB^4(\Omega)$, Lemma \ref{lem:growth-condition} and Proposition \ref{prop:linf-h2}.
\end{proof}

\begin{proposition} \label{t-nabla-mk-proposition}
For each $k \in \N$ and $t \in [0,T^*]$,
\begin{equation*}
    \int_0^t \lnorm{\partial_s \nabla \vec{m}_k(s,\cdot)}{2}^2 \; \ds \lsim 1.
\end{equation*}
\end{proposition}
\begin{proof}
Setting $\boldsymbol{\varphi} = -\partial_t \Delta \vec{m}_k$ in (\ref{galerkin-eqn}), we quickly see that
\begin{align} \label{time-deriv-h1-ineq}
    \begin{split}
        \lnorm{\partial_t \nabla \vec{m}_k}{2}^2 + \frac{1}{2} \ddt \lnorm{\Delta \vec{m}_k}{2}^2 &\lsim \left|\liprod{\nabla\left(|\nabla \vec{m}_k|^2 \vec{m}_k\right)}{\partial_t \nabla \vec{m}_k}\right| \\
        &\qquad + \left|\liprod{\nabla(\vec{m}_k \times \Delta \vec{m}_k)}{\partial_t \nabla \vec{m}_k}\right| \\
        &\qquad + \left|\liprod{\nabla \vec{g}_k}{\partial_t \nabla \vec{m}_k}\right| \\[1ex]
        &\lsim J_1 + J_2 + J_3,
    \end{split}
\end{align}
where
\begin{align*}
    J_1 &= \intomega{|\nabla\left(|\nabla \vec{m}_k|^2 \vec{m}_k\right) \cdot \partial_t \nabla \vec{m}_k|}, \\
    J_2 &= \intomega{|\nabla(\vec{m}_k \times \Delta \vec{m}_k) \cdot \partial_t \nabla \vec{m}_k|}, \\
    J_3 &= \intomega{|\nabla \vec{F}(\vec{m}_k, \nabla \vec{m}_k) \cdot \partial_t \nabla \vec{m}_k|}.
\end{align*}
These terms are bounded similarly to $I_1, I_2$ and $I_3$ in Proposition \ref{prop:linf-h2} with slight modifications to obtain
\begin{equation*}
    \lnorm{\partial_t \nabla \vec{m}_k}{2}^2 + \frac{1}{2} \ddt \lnorm{\Delta \vec{m}_k}{2}^2 \lsim 1 + \lnorm{\nabla \Delta \vec{m}_k}{2}^2 + \varepsilon^2 \lnorm{\partial_t \nabla \vec{m}_k}{2}^2.
\end{equation*}
For sufficiently small $\varepsilon$ we conclude that
\begin{align*}
    \int_0^t \lnorm{\partial_s \nabla \vec{m}_k(s,\cdot)}{2}^2 \; \ds + \lnorm{\Delta \vec{m}(T^*,\cdot)}{2}^2  &\lsim \lnorm{\Delta \vec{m}_k(0,\cdot)}{2}^2 + \int_0^t 1 + \lnorm{\nabla \Delta \vec{m}_k(s,\cdot)}{2}^2 \; \ds.
\end{align*}
The result follows from Proposition \ref{prop:linf-h2}.
\end{proof}

\subsection{Proof of Theorem \ref{exist-thm}}

It follows from Propositions \ref{prop:linf-h2}-\ref{t-nabla-mk-proposition} and the Banach-Alaoglu theorem that there exists a subsequence $\{\vec{m}_k\}$, still denoted by $\{\vec{m}_k\}$, such that
\begin{equation*}
    \begin{cases}
        \vec{m}_k \weakto \vec{m} & \text{in } L^2(0,T^*;\HB^3(\Omega)), \\[1ex]
        \vec{m}_k \wkstarto \vec{m} & \text{in } L^{\infty}(0,T^*;\HB^2(\Omega)), \\[1ex]
        \partial_t \vec{m}_k \weakto \partial_t \vec{m} & \text{in } L^2(0,T^*; \HB^1(\Omega)).
    \end{cases}
\end{equation*}
By Aubin's Lemma (Theorem \ref{aubins-lemma}), we conclude that
\begin{equation*}
    \vec{m}_k \to \vec{m} \text{ strongly in } L^2(0,T^*;\HB^2(\Omega)).
\end{equation*}
Then by Theorem \ref{time-continuity-theorem} we have
\begin{equation*}
    \vec{m} \in C([0,T^*]; \HB^2(\Omega)).
\end{equation*}

Having obtained these results, we can now show the convergence of the non-linear terms in (\ref{galerkin-eqn}).

\begin{lemma} \label{nonlin-convergence-proposition}
Let $T^* > 0$ be as defined in Proposition \ref{prop:linf-h2}. For any $\boldsymbol{\psi} \in C^{\infty}(\Omega)$ and each $t \in [0,T^*]$ we have
\begin{align*}
    \lim_{k \to \infty} \int_0^t \liprod{|\nabla \vec{m}_k|^2 \vec{m}_k}{\boldsymbol\psi} \; \ds &= \int_0^t \liprod{|\nabla \vec{m}|^2 \vec{m}}{\boldsymbol\psi} \; \ds \\
    \lim_{k \to \infty} \intOt{\liprod{\vec{m}_k \times \Delta \vec{m}_k}{\boldsymbol\psi}} &= \intOt{\liprod{\vec{m} \times \Delta \vec{m}}{\boldsymbol\psi}} \\
    \lim_{k \to \infty} \intOt{\liprod{\vec{F}(\vec{m}_k,\nabla \vec{m}_k)}{\boldsymbol\psi}} &= \intOt{\liprod{\vec{F}(\vec{m},\nabla \vec{m})}{\boldsymbol\psi}}
\end{align*}
\end{lemma}
\begin{proof}
For the first result, we see that

\begin{align*}
    &\left|\int_0^t \liprod{|\nabla \vec{m}_k|^2 \vec{m}_k}{\boldsymbol\psi} \; \ds - \int_0^t \liprod{|\nabla \vec{m}|^2 \vec{m}}{\boldsymbol\psi} \; \ds\right| \\[1ex]
    &\qquad \lsim \left(\intOt{\lnorm{|\nabla \vec{m}_k|^2 \vec{m}_k - |\nabla \vec{m}|^2 \vec{m}}{2}^2}\right)^{1/2} \\[1ex]
    &\qquad \lsim \left(\intOt{\lnorm{|\nabla \vec{m}_k|^2 \vec{m}_k - |\nabla \vec{m}_k|^2 \vec{m}}{2}^2}\right)^{1/2} \\
    &\qquad \qquad + \left(\intOt{\lnorm{|\nabla \vec{m}_k|^2 \vec{m} - |\nabla \vec{m}|^2 \vec{m}}{2}^2}\right)^{1/2}.
\end{align*}

Now we proceed by showing that the first integral tends to zero using Proposition \ref{prop:linf-h2} and strong convergence in $L^2(0,T^*; \HB^2(\Omega))$. That is,
\begin{align*}
    \intOt{\lnorm{|\nabla \vec{m}_k|^2 \vec{m}_k - |\nabla \vec{m}_k|^2 \vec{m}}{2}^2} &\lsim \intOt{\lnorm{\vec{m}_k - \vec{m}}{\infty}^2 \lnorm{\nabla \vec{m}_k}{4}^4}\\
    &\lsim \intOt{\hnorm{\vec{m}_k - \vec{m}}{2}^2} \to 0,
\end{align*}
Similarly for the second term,
\begin{align*}
    &\intOt{\lnorm{|\nabla \vec{m}_k|^2 \vec{m} - |\nabla \vec{m}|^2 \vec{m}}{2}^2} \\[1ex]
    &\qquad \lsim \int_0^t \lnorm{\vec{m}}{\infty}^2 \intomega{|\nabla \vec{m}_k - \nabla \vec{m}|^2 \left(|\nabla \vec{m}_k|^2 + |\nabla \vec{m}|^2\right)} \\
    &\qquad \lsim \intOt{\lnorm{\nabla \vec{m}_k - \nabla \vec{m}}{4}^2 \lnorm{\nabla \vec{m}_k}{4}^2} + \intOt{\lnorm{\nabla \vec{m}_k - \nabla \vec{m}}{4}^2 \lnorm{\nabla \vec{m}}{4}^2} \\
    &\qquad \lsim \intOt{\lnorm{\nabla \vec{m}_k - \nabla \vec{m}}{4}^2}, \\
    &\qquad \lsim \intOt{\hnorm{\vec{m}_k - \vec{m}}{2}^2} \to 0.
\end{align*}
The first result now follows.

As the method remains the same with slight modifications, we omit the some of the more obvious steps for the next two results. For the second result,
\begin{align*}
    \Bigg|\intOt{\liprod{\vec{m}_k \times \Delta \vec{m}_k}{\boldsymbol\psi}} &- \intOt{\liprod{\vec{m} \times \Delta \vec{m}}{\boldsymbol\psi}}\Bigg| \\
    &\lsim \left(\intOt{\lnorm{\vec{m}_k \times \Delta \vec{m}_k - \vec{m} \times \Delta \vec{m}_k}{2}^2}\right)^{1/2} \\
    &\qquad + \left(\intOt{\lnorm{\vec{m} \times \Delta \vec{m}_k - \vec{m} \times \Delta \vec{m}}{2}^2}\right)^{1/2}.
\end{align*}

It is easy to see that the these two integral tend to zero as
\begin{align*}
    \intOt{\lnorm{\vec{m}_k \times \Delta \vec{m}_k - \vec{m} \times \Delta \vec{m}_k}{2}^2} &\lsim \intOt{\lnorm{\vec{m}_k -\vec{m}}{\infty}^2 \lnorm{\Delta \vec{m}_k}{2}^2} \\
    &\lsim \intOt{ \hnorm{\vec{m}_k - \vec{m}}{2}^2} \to 0,
\end{align*}
and
\begin{align*}
    \intOt{\lnorm{\vec{m} \times \Delta \vec{m}_k - \vec{m} \times \Delta \vec{m}}{2}^2} &\lsim \intOt{\lnorm{\vec{m}}{\infty}^2 \lnorm{\Delta \vec{m}_k - \Delta \vec{m}}{2}^2} \\
    &\lsim \intOt{\hnorm{\vec{m}_k - \vec{m}}{2}^2} \to 0,
\end{align*}
proving the second result.

For the non-homogeneous term, we first have that
\begin{align*}
    \Bigg|\intOt{\liprod{\vec{F}(\vec{m}_k,\nabla \vec{m}_k)}{\boldsymbol\psi}} &- \intOt{\liprod{\vec{F}(\vec{m},\nabla \vec{m})}{\boldsymbol\psi}}\Bigg| \\
    &\lsim \left(\intOt{\lnorm{\vec{F}(\vec{m}_k,\nabla \vec{m}_k) - \vec{F}(\vec{m},\nabla \vec{m})}{2}^2}\right)^{1/2}.
\end{align*}
But since $\vec{m}_k$ and $\vec{m}$ belong to a bounded set in $\HB^2(\Omega)$, applying Lemma \ref{lem:locally-lipschitz} gives
\begin{align*}
    \intOt{\lnorm{\vec{F}(\vec{m}_k,\nabla \vec{m}_k) - \vec{F}(\vec{m},\nabla \vec{m})}{2}^2} &\lsim \intOt{\hnorm{\vec{m}_k - \vec{m}}{1}^2} \\
    &\lsim \intOt{\hnorm{\vec{m}_k - \vec{m}}{2}^2} \to 0,
\end{align*}
giving us the third result.
\end{proof}

Consequently, $\vec{m}$ satisfies the weak formulation in Definition \ref{def-weak-solution-theoretical}. Next, we note that the solution $\vec{m}$ satisfies the Neumann boundary condition. Indeed, by using the Trace Theorem, we have
\begin{align*}
    \int_0^t \left\|\frac{\partial \vec{m}}{\partial\vec{n}}\right\|_{\LB^2(\partial\Omega)}^2 \; \ds &\lsim \int_0^t \left\|\frac{\partial \vec{m}}{\partial\vec{n}} - \frac{\partial \vec{m}_k}{\partial\vec{n}}\right\|_{\LB^2(\partial\Omega)}^2 \; \ds + \int_0^t \left\|\frac{\partial \vec{m}_k}{\partial\vec{n}}\right\|_{\LB^2(\partial\Omega)}^2 \; \ds \\
    &= \int_0^t \left\|\frac{\partial \vec{m}}{\partial\vec{n}} - \frac{\partial \vec{m}_k}{\partial\vec{n}}\right\|_{\LB^2(\partial\Omega)}^2 \; \ds \\
    &\lsim \; \int_0^t \hnorm{\vec{m} - \vec{m}_k}{2}^2 \; \ds \to 0.
\end{align*}

Finally, by following the proof in \cite{carbou2001-article}, using the orthogonality property of $\vec{F}$ i.e., Assumption (A2), we conclude that $\vec{m}$ has magnitude $|\vec{m}| = 1$. Moreover, because $\vec{m}_k(0,\cdot)$ is the projection of $\vec{m}_0$ onto $V_k$, we have $\vec{m}(0,\cdot) = \vec{m}_0$. Therefore $\vec{m}$ is a solution in the sense of Definition \ref{def-weak-solution-theoretical}, completing the proof of Theorem \ref{exist-thm}.

\subsection{Proof of Theorem \ref{unique-thm}}

Suppose $\vec{m}_1$ and $\vec{m}_2$ are two solutions obtained from Theorem \ref{exist-thm} with initial data $\vec{m}_{0,1}$ and $\vec{m}_{0,2}$, respectively. Let $T' = \min(T_1^*, T_2^*)$ and $\vec{v} = \vec{m}_1 - \vec{m}_2$. Then $\vec{v}$ solves the following PDE almost everywhere on $(0,T') \times \Omega$:
\begin{align*}
    \partial_t \vec{v} - \beta'\Delta \vec{v} &= \beta'|\nabla \vec{m}_1|^2 \vec{v} + \beta'(|\nabla \vec{m}_1|^2 - |\nabla \vec{m}_2|^2) \vec{m}_2 + \alpha'\vec{v} \times \Delta \vec{m}_1 + \alpha'\vec{m}_2 \times \Delta \vec{v} \\
    &\qquad + \vec{F}(\vec{m}_1, \nabla \vec{m}_2) - \vec{F}(\vec{m}_2, \nabla \vec{m}_2).
\end{align*}
Taking the dot product with $\vec{v}$, we obtain
\begin{equation} \label{diff-inner-prod}
    \frac{1}{2} \ddt \lnorm{\vec{v}}{2}^2 + \beta'\lnorm{\nabla \vec{v}}{2}^2 = \alpha'K_1 + \beta'K_2 + \beta'K_3 + K_4,
\end{equation}
where
\begin{align*}
    K_1 &= \intomega{(\vec{m}_2 \times \Delta \vec{v}) \cdot \vec{v}}, \\
    K_2 &= \intomega{|\nabla \vec{m}_1|^2 \vec{v} \cdot \vec{v}}, \\
    K_3 &= \intomega{(|\nabla \vec{m}_1|^2 - |\nabla \vec{m}_2|^2) \vec{m}_2 \cdot \vec{v}}, \\
    K_4 &= \intomega{(\vec{F}(\vec{m}_1, \nabla \vec{m}_1) - \vec{F}(\vec{m}_2, \nabla \vec{m}_2)) \cdot \vec{v}}.
\end{align*}
We now bound each of these integrals. We obtain by using integration by parts, Holder's inequality, and Sobolev embedding
\begin{equation*}
    |K_1| \lsim \lnorm{\nabla \vec{m}_2}{\infty} \lnorm{\nabla \vec{v}}{2} \lnorm{\vec{v}}{2}
    \lsim \hnorm{\vec{m}_2}{3} \lnorm{\nabla \vec{v}}{2} \lnorm{\vec{v}}{2}
\end{equation*}
The bound for $K_2$ is straightforward and is just
\begin{equation*}
    |K_2| \leq \lnorm{\nabla \vec{m}_1}{\infty}^2 \lnorm{\vec{v}}{2}^2 \lsim \hnorm{\vec{m}_1}{3}^2 \lnorm{\vec{v}}{2}^2
\end{equation*}
For~$K_3$, by using $|\vec{a}|^2 - |\vec{b}|^2 = (|\vec{a}| - |\vec{b}|)(|\vec{a}| + |\vec{b}|)$ we have
\begin{align*}
    |K_3| &\lsim \left(\lnorm{\nabla \vec{m}_1}{\infty} + \lnorm{\nabla \vec{m}_2}{\infty}\right) \lnorm{\nabla \vec{v}}{2} \lnorm{\vec{v}}{2} \\
    &\lsim \left(\hnorm{\vec{m}_1}{3} + \hnorm{\vec{m}_2}{3}\right) \lnorm{\nabla \vec{v}}{2} \lnorm{\vec{v}}{2}.
\end{align*}
Finally, by using Lemma \ref{lem:locally-lipschitz} we have
\begin{align*}
    |K_4| &\leq \lnorm{\vec{F}(\vec{m}_1, \nabla \vec{m}_1) - \vec{F}(\vec{m}_2, \nabla \vec{m}_2)}{2} \lnorm{\vec{v}}{2} \\
    &\lsim \hnorm{\vec{m}_1 - \vec{m}_2}{1} \lnorm{\vec{v}}{2} \\
    &\lsim \lnorm{\vec{v}}{2}^2 + \lnorm{\nabla \vec{v}}{2} \lnorm{\vec{v}}{2}.
\end{align*}
Therefore, it follows from~\eqref{diff-inner-prod} that
\begin{align*}
    \frac{1}{2} \ddt \lnorm{\vec{v}}{2}^2 + \beta'\lnorm{\nabla \vec{v}}{2}^2 &\lsim \hnorm{\vec{m}_2}{3} \lnorm{\nabla \vec{v}}{2} \lnorm{\vec{v}}{2} + \hnorm{\vec{m}_1}{3}^2 \lnorm{\vec{v}}{2}^2 \\
    &\qquad+ \left(\hnorm{\vec{m}_1}{3} + \hnorm{\vec{m}_2}{3}\right) \lnorm{\nabla \vec{v}}{2} \lnorm{\vec{v}}{2} \\
    &\qquad+ \lnorm{\vec{v}}{2}^2 + \lnorm{\nabla \vec{v}}{2} \lnorm{\vec{v}}{2}.
\end{align*}
Using Young's inequality, we can absorb the term involving $\lnorm{\nabla \vec{v}}{2}^2$ to the corresponding term on the left hand side and obtain
\begin{align*}
    \ddt \lnorm{\vec{v}}{2}^2 + \lnorm{\nabla \vec{v}}{2}^2 &\lsim \left(\hnorm{\vec{m}_1}{3}^2 + \hnorm{\vec{m}_2}{3}^2 + 1\right) \lnorm{\vec{v}}{2}^2.
\end{align*}
As $\vec{m}_1, \vec{m}_2 \in L^2(0,T'; \HB^3(\Omega))$, the conclusion follows from the Gronwall Lemma.

\section{Numerical Method} \label{sec:numerical-method}

It follows from \eqref{eq:llg-numerical-pde}, that for all $\boldsymbol{\psi} \in C^{\infty}_0(\Omega_T)$
\begin{align} \label{eq:llg-numerical-pde-weak-formulation}
\begin{split}
    \beta \liprodt{\partial_t \vec{m}}{\boldsymbol{\psi}} - \alpha \liprodt{\vec{m} \times \partial_t \vec{m}}{\boldsymbol{\psi}} &= \alpha \liprodt{|\nabla \vec{m}|^2 \vec{m}}{\boldsymbol{\psi}} - \alpha \liprodt{\nabla \vec{m}}{\nabla \boldsymbol{\psi}} \\
    &\quad - \liprodt{\vec{m} \times \vec{f}\left(\vec{m}, \nabla \vec{m}\right)}{\boldsymbol{\psi}}.
\end{split}
\end{align}
We will use the finite element method to approximate the solution~$\vec{m}$ of this equatiaon. We define the finite element space $\V_h \subset \HB^1(\Omega)$ to be the space of all continuous piecewise linear functions on a triangulation of $\Omega$. Let $\{\phi_n\}_{n=1}^N$ be the hat functions satisfying $\phi_n(\vec{x}_m) = \delta_{n,m}$, where $\delta$ is the Kronecker delta. Thus we can define the interpolation operator to be
\begin{equation*}
    I_{\V_h}(\vec{u}) = \sum_{n=1}^N \vec{u}(\vec{x}_n) \phi_n(\vec{x}),
\end{equation*}
for all $\vec{u} \in C(\Omega; \R^3)$. Let $k$ denote the time-step and $j$ the increment of time i.e. $t_j = jk$. Denote $\femvec{m}{h}{j} \in \V_h$ to be approximation of $\vec{m}(t_j,\cdot)$ and $\femvec{v}{h}{j}$ to be the approximation for $\partial_t \vec{m}(t_j,\cdot)$. Inspired by the fact that $\vec{m} \cdot \partial_t \vec{m} = 0$, we search for $\femvec{v}{h}{j}$ in the space
\begin{equation*}
    \W_h^{(j)} := \left\{\vec{w} \in \V_h : \vec{w}(\vec{x}_n) \cdot \femvec{m}{h}{j}(\vec{x}_n) = 0, \; 1 \leq n \leq N\right\}.
\end{equation*}
Using the formulation \eqref{eq:llg-numerical-pde-weak-formulation} we extend the algorithm in \cite{alouges2008-article} to get the following.

\subsection*{Algorithm 3.1}
Choose $\theta \in (\frac{1}{2}, 1]$ and a time-step size $k = \frac{T}{J}$ with $J \in \N$. \\
\noindent \textbf{Step 1}: Set $j = 0$. Choose $\femvec{m}{h}{0} = I_{\V_h} (\vec{m}_0)$. \\
\noindent \textbf{Step 2}: Find $\femvec{v}{h}{j} \in \W_h^{(j)}$ such that for all $\boldsymbol{\psi} \in \W_h^{(j)}$
\begin{align} \label{section-3-fem-equation}
    \begin{split}
        \beta \liprod{\femvec{v}{h}{j}}{\boldsymbol\psi} - \liprod{\femvec{m}{h}{j} \times \femvec{v}{h}{j}}{\boldsymbol{\psi}} &= -\alpha \liprod{\nabla \left(\femvec{m}{h}{j} + \theta k \femvec{v}{h}{j}\right)}{\nabla \boldsymbol{\psi}} \\
        &\qquad - \liprod{\femvec{m}{h}{j} \times \vec{f}\left(\femvec{m}{h}{j}, \nabla \femvec{m}{h}{j}\right)}{\boldsymbol{\psi}}.
    \end{split}
\end{align}
\noindent \textbf{Step 3}: Define
\begin{equation*}
    \femvec{m}{h}{j+1}(\vec{x}) := \sum_{n=1}^N \frac{\femvec{m}{h}{j}(\vec{x}_n) + k \femvec{v}{h}{j}(\vec{x}_n)}{\left|\femvec{m}{h}{j}(\vec{x}_n) + k \femvec{v}{h}{j}(\vec{x}_n)\right|} \phi_n(\vec{x}).
\end{equation*}
\noindent \textbf{Step 4}: Set $j = j+1$, and return to Step 2 if $j < J$. Stop if $j = J$. \\

\begin{theorem}[Numerical Convergence] \label{thm:numerical}
Let $T > 0$ and assume that $\vec{m}_0 \in \HB^1(\Omega)$ satisfies $|\vec{m}_0| \equiv 1$ and that $\vec{f}$ satisfies assumptions (A1) and (A2). Suppose that $\femvec{m}{h}{0} \to \vec{m}_0$ in $\HB^1(\Omega)$ as $h \to 0$, and $\theta \in (\frac{1}{2}, 1]$. If $d=2,3$ and the triangulation $\T_h$ satisfies (\ref{section-2-bartels-condition}) then as $(h,k) \to (0,0)$, $\vec{m}_{h,k}$ converges weakly in $\HB^1(\Omega_T)$ and strongly in $\LB^2(\Omega_T)$ to a weak solution of \eqref{eq:llg-numerical-pde-weak-formulation}.
\end{theorem}

\subsection{Proof of Theorem \ref{thm:numerical}}
\begin{lemma} \label{linf-bound}
For each $j \in \{0, \ldots, J\}$,
\begin{equation*}
    \lnorm{\femvec{m}{h}{j}}{\infty} \lsim 1 \quad \text{and} \quad \lnorm{\femvec{m}{h}{j}}{2} \lsim |\Omega|,
\end{equation*}
where $|\Omega|$ is the measure of $\Omega$.
\end{lemma}
\begin{proof}
The proof can be found in \cite[Lemma 5.2]{goldys2016-article}.
\end{proof}

As our numerical construction bounds the $\LB^{\infty}(\Omega)$ norm of $\vec{m}_h^{(j)}$, we can weaken the result of Lemma~\ref{lem:growth-condition} for $\vec{f}$.

\begin{lemma} \label{lem:f-linf-bound}
Let $\vec{u} \in \LB^{\infty}(\Omega) \cap \HB^1(\Omega)$ be such that $\vec{u} \in B_R(\vec{0}) \subset \LB^{\infty}(\Omega)$ where $B_R(\vec{0})$ is the ball centred at $\vec{0}$ with radius $R$. Then
\begin{equation*}
    \lnorm{\vec{f}(\vec{u}, \nabla \vec{u})}{2} \lsim 1 + \lnorm{\nabla \vec{u}}{2},
\end{equation*}
where the implicit constant may depend on $R$.
\end{lemma}
\begin{proof}
Taking the $\LB^2(\Omega)$ of $\vec{f}(\vec{u}, \nabla \vec{u})$ we obtain
\begin{align*}
    \lnorm{\vec{f}(\vec{u}, \nabla \vec{u})}{2} &\lsim \lnorm{\vec{f}_1(\vec{u})}{2} + \lnorm{\vec{f}_2(\vec{u})}{\infty} \lnorm{\vec{g}_1(\nabla \vec{u})}{2}\\
    &\qquad + \lnorm{\vec{f}_3(\vec{u})}{\infty} \lnorm{\vec{f}_4(\vec{u})}{\infty} \lnorm{\vec{g}_2(\nabla \vec{u})}{2} \\
    &\lsim 1 + \lnorm{\nabla \vec{u}}{2},
\end{align*}
using Assumption (A2) on the growth conditions on $\vec{f}_i$, for $i=1,2,3,4$, and Assumption (A1) on the linearity of $\vec{g}_j$, for $j=1,2$.
\end{proof}

Using this lemma, we can now start to bound the gradient and time-derivatives of $\vec{m}_h^{(j)}$.

\begin{lemma} \label{discrete-bounds}
For all $j \in \{1, \ldots, J\}$,
\begin{equation*}
    \lnorm{\nabla \femvec{m}{h}{j}}{2}^2 + \sum_{i=0}^{j-1} k \lnorm{\femvec{v}{h}{j}}{2}^2 + k^2(2 \theta -1) \sum_{i=0}^{j-1} \lnorm{\nabla \femvec{v}{h}{j}}{2}^2 \lsim 1.
\end{equation*}
\end{lemma}
\begin{proof}
We return to the finite element weak formulation (\ref{section-3-fem-equation})
\begin{align*}
    \beta \liprod{\femvec{v}{h}{j}}{\boldsymbol\psi} -  \liprod{\femvec{m}{h}{j} \times \femvec{v}{h}{j}}{\boldsymbol{\psi}} &= -\alpha \liprod{\nabla \left(\femvec{m}{h}{j} + \theta k \femvec{v}{h}{j}\right)}{\nabla \boldsymbol{\psi}} \\
    &\qquad - \liprod{\femvec{m}{h}{j} \times \vec{f}\left(\femvec{m}{h}{j}, \nabla \femvec{m}{h}{j}\right)}{\boldsymbol{\psi}}.
\end{align*}
Setting $\boldsymbol{\psi} = \femvec{v}{h}{j}$ we see that
\begin{equation*}
    \beta \lnorm{\femvec{v}{h}{j}}{2}^2 = -\alpha \liprod{\nabla \femvec{m}{h}{j}}{\nabla \femvec{v}{h}{j}} - \alpha \theta k \lnorm{\nabla \femvec{v}{h}{j}}{2}^2 - \liprod{\femvec{m}{h}{j} \times \vec{f}\left(\femvec{m}{h}{j},\nabla \femvec{m}{h}{j}\right)}{\femvec{v}{h}{j}}
\end{equation*}
and so
\begin{equation*}
    \liprod{\nabla \femvec{m}{h}{j}}{\nabla \femvec{v}{h}{j}} = -\frac{\beta}{\alpha} \lnorm{\femvec{v}{h}{j}}{2}^2 - \theta k \lnorm{\nabla \femvec{v}{h}{j}}{2}^2 - \frac{\beta}{\alpha} \liprod{\femvec{m}{h}{j} \times \vec{f}\left(\femvec{m}{h}{j},\nabla \femvec{m}{h}{j}\right)}{\femvec{v}{h}{j}}.
\end{equation*}
By Theorem \ref{section-2-bartels-theorem},
\begin{equation*}
    \lnorm{\nabla \femvec{m}{h}{j+1}}{2}^2 \leq \lnorm{\nabla \left(\femvec{m}{h}{j} + k \femvec{v}{h}{j}\right)}{2}^2
\end{equation*}
therefore 
\begin{align*}
    \lnorm{\nabla \femvec{m}{h}{j+1}}{2}^2 &\leq \lnorm{\nabla \femvec{m}{h}{j}}{2}^2 + k^2 \lnorm{\nabla \femvec{v}{h}{j}}{2}^2 + 2 k \liprod{\nabla \femvec{m}{h}{j}}{\nabla \femvec{v}{h}{j}} \\[1ex]
    &=  \lnorm{\nabla \femvec{m}{h}{j}}{2}^2 + k^2 \lnorm{\nabla \femvec{v}{h}{j}}{2}^2 -\frac{2\beta k}{\alpha} \lnorm{\femvec{v}{h}{j}}{2}^2 - 2\theta k^2 \lnorm{\nabla \femvec{v}{h}{j}}{2}^2 \\
    &\qquad - \frac{2\beta k}{\alpha} \liprod{\femvec{m}{h}{j} \times \vec{f}\left(\femvec{m}{h}{j},\nabla \femvec{m}{h}{j}\right)}{\femvec{v}{h}{j}} \\[1ex]
    &\leq \lnorm{\nabla \femvec{m}{h}{j}}{2}^2 + k^2 (1-2\theta) \lnorm{\nabla \femvec{v}{h}{j}}{2}^2 -\frac{2k}{\alpha} \lnorm{\femvec{v}{h}{j}}{2}^2 \\
    &\qquad \frac{k}{\alpha \beta} \lnorm{\vec{f}\left(\femvec{m}{h}{j},\nabla \femvec{m}{h}{j}\right)}{2}^2 + \frac{\beta k}{\alpha}\lnorm{\femvec{v}{h}{j}}{2}^2.
\end{align*}
Summing over the time periods $0,\ldots,j$ and using Lemma \ref{linf-bound} and \ref{lem:f-linf-bound} we obtain
\begin{align} \label{sugiyama-setup}
\begin{split}
    \lnorm{\nabla \femvec{m}{h}{j}}{2}^2 + \sum_{i=0}^{j-1} k \lnorm{\femvec{v}{h}{j}}{2}^2 &+ k^2(2 \theta -1) \sum_{i=0}^{j-1} \lnorm{\nabla \femvec{v}{h}{j}}{2}^2 \\
    &\lsim \lnorm{\nabla \femvec{m}{h}{0}}{2}^2 + k \sum_{i=0}^{j-1} \lnorm{\vec{f}\left(\femvec{m}{h}{i},\nabla \femvec{m}{h}{i}\right)}{2}^2. \\
    &\lsim 1 + C k \sum_{i=0}^{j-1} \lnorm{\nabla \femvec{m}{h}{i}}{2}^2.
\end{split}
\end{align}
By the discrete Gronwall inequality \cite{sugiyama1969-article}
\begin{equation*}
    \lnorm{\nabla \femvec{m}{h}{j}}{2}^2 \lsim (1 + Ck)^i,
\end{equation*}
and so by summing over $i = 0, \ldots, j-1$ and using $1 + x \leq e^x$ we have
\begin{equation*}
    k \sum_{i=0}^{j-1} \lnorm{\nabla \femvec{m}{h}{i}}{2}^2 \lsim k \frac{(1 + Ck)^j - 1}{Ck} \lsim e^{CkJ} \lsim 1.
\end{equation*}
Combining this with (\ref{sugiyama-setup}) gives us the required result.
\end{proof}

\begin{definition}
For all $\vec{x} \in \Omega$ and all $t \in [0,T]$, let $j \in {0, \ldots, J}$ be such that $t \in [t_j, t_{j+1})$. We then define
\begin{align*}
    \vec{m}_{h,k}(t,\vec{x}) &:= \frac{t-t_j}{k} \femvec{m}{h}{j+1}(\vec{x}) + \frac{t_{j+1} - t}{k} \femvec{m}{h}{j}(\vec{x}), \\
    \vec{m}_{h,k}^- (t,\vec{x}) &:= \femvec{m}{h}{j}(\vec{x}), \\
    \vec{v}_{h,k}(t,\vec{x}) &:= \femvec{v}{h}{j}(\vec{x}).
\end{align*}
\end{definition}

\begin{lemma} \label{interpolant-bounds}
For all $\theta \in (\frac{1}{2},1]$,
\begin{equation*}
    \lnormt{\vec{m}_{h,k}^*}{2}^2 + \lnormt{\nabla \vec{m}_{h,k}^*}{2}^2 + \lnormt{\vec{v}_{h,k}}{2}^2 + k(2\theta-1) \lnormt{\nabla \vec{v}_{h,k}}{2}^2 \lsim 1,
\end{equation*}
where $\vec{m}_{h,k}^* = \vec{m}_{h,k}$ or $\vec{m}_{h,k}^-$.
\end{lemma}
\begin{proof}
The proof can be found in \cite[Lemma 6.2]{goldys2016-article}
\end{proof}

\begin{lemma} \label{discrete-interpolant-error}
Assume that $h$ and $k$ go to 0. The sequences $\{\vec{m}_{h,k}\}$, $\{\vec{m}_{h,k}^-\}$ and $\{\vec{v}_{h,k}\}$ satisfy the properties
\begin{align}
    \hnormt{\vec{m}_{h,k}}{1} &\lsim 1, \label{discrete-interpolant-error-1} \\
    \lnormt{\vec{m}_{h,k} - \vec{m}_{h,k}^-}{2} &\lsim k, \label{discrete-interpolant-error-2} \\
    \lnormt{\vec{v}_{h,k} - \partial_t \vec{m}_{h,k}}{1} &\lsim k, \label{discrete-interpolant-error-3} \\
    \lnormt{|\vec{m}_{h,k}| - 1}{2} &\lsim h+k \label{discrete-interpolant-error-4}.
\end{align}
\end{lemma}
\begin{proof}
These facts have already been proved in \cite{alouges2008-article}.
\end{proof}

We now write the weak formulation of the numerical solutions with test functions in $C_0^{\infty}(\Omega_T)$ instead of $\V_h$ as a precusor to showing limiting solution satisfies the weak formulation \eqref{eq:llg-numerical-pde-weak-formulation}.

\begin{lemma} \label{section-4-discrete-weak-formulation}
For any $\boldsymbol{\varphi} \in C_0^{\infty}(\Omega_T)$, as $h$ and $k$ go to zero,
\begin{align*}
    \begin{split}
        &\beta \liprodt{\vec{v}_{h,k}}{\vec{m}_{h,k}^- \times \boldsymbol{\varphi}} - \liprodt{\vec{m}_{h,k}^- \times \vec{v}_{h,k}}{\vec{m}_{h,k}^- \times \boldsymbol{\varphi}} \\
        &+ \alpha \liprodt{\nabla (\vec{m}_{h,k}^- + \theta k \vec{v}_{h,k})}{\nabla (\vec{m}_{h,k}^- \times \boldsymbol{\varphi})} + \liprodt{\vec{m}_{h,k}^- \times \vec{f}(\vec{m}_{h,k}^-, \nabla \vec{m}_{h,k}^-)}{\vec{m}_{h,k}^- \times \boldsymbol{\varphi}}  = O(h).
    \end{split}
\end{align*}
\end{lemma}
\begin{proof}
Following the argument in \cite[Lemma 6.4]{goldys2016-article}, we deduce that
\begin{align*}
    \begin{split}
        \beta \liprodt{\vec{v}_{h,k}}{\vec{m}_{h,k}^- \times \boldsymbol{\varphi}} &- \liprodt{\vec{m}_{h,k}^- \times \vec{v}_{h,k}}{\vec{m}_{h,k}^- \times \boldsymbol{\varphi}} \\
        &+ \alpha \liprodt{\nabla (\vec{m}_{h,k}^- + \theta k \vec{v}_{h,k})}{\nabla (\vec{m}_{h,k}^- \times \boldsymbol{\varphi})} \\
        &+ \liprodt{\vec{m}_{h,k}^- \times \vec{f}(\vec{m}_{h,k}^-, \nabla \vec{m}_{h,k}^-)}{\vec{m}_{h,k}^- \times \boldsymbol{\varphi}} \\
        &\qquad = I_1 + I_2 + I_3 + I_4,
    \end{split}
\end{align*}
where
\begin{align*}
    I_1 &= \liprodt{\beta \vec{v}_{h,k}}{\vec{m}_{h,k}^- \times \boldsymbol{\varphi} - I_{\V_h}(\vec{m}_{h,k}^- \times \boldsymbol{\varphi})}, \\
    I_2 &= \liprodt{- \vec{m}_{h,k}^- \times \vec{v}_{h,k}}{\vec{m}_{h,k}^- \times \boldsymbol{\varphi} - I_{\V_h}(\vec{m}_{h,k}^- \times \boldsymbol{\varphi})}, \\
    I_3 &= \alpha \liprodt{\nabla (\vec{m}_{h,k}^- + \theta k \vec{v}_{h,k})}{\nabla (\vec{m}_{h,k}^- \times \boldsymbol{\varphi} - I_{\V_h}(\vec{m}_{h,k}^- \times \boldsymbol{\varphi}))}, \\
    I_4 &= \liprodt{\vec{m}_{h,k}^- \times \vec{f}(\vec{m}_{h,k}^-, \nabla \vec{m}_{h,k}^-)}{\vec{m}_{h,k}^- \times \boldsymbol{\varphi} - I_{\V_h}(\vec{m}_{h,k}^- \times \boldsymbol{\varphi})}.
\end{align*}
The integrals $I_1, I_2$ and $I_3$ have already been bounded in \cite{goldys2016-article}, so we focus on $I_4$. Using their argument and Lemmas \ref{lem:f-linf-bound} and \ref{interpolant-bounds} we obtain
\begin{align*}
    |I_4| &\lsim \lnormt{\vec{m}_{h,k}^-}{\infty} \lnormt{\vec{f}(\vec{m}_{h,k}^-, \nabla \vec{m}_{h,k}^-)}{2} \lnormt{\vec{m}_{h,k}^- \times \boldsymbol{\varphi} - I_{\V_h}(\vec{m}_{h,k}^- \times \boldsymbol{\varphi})}{2} \\
    &\lsim  \lnormt{\vec{m}_{h,k}^- \times \boldsymbol{\varphi} - I_{\V_h}(\vec{m}_{h,k}^- \times \boldsymbol{\varphi})}{2} \\
    &\lsim h,
\end{align*}
thus completing the proof.
\end{proof}

\begin{lemma} \label{section-4-interpolant-weak-formulation}
For any $\boldsymbol{\varphi} \in C_0^{\infty}(\Omega_T)$, as $h$ and $k$ go to zero,
\begin{align*}
    \begin{split}
        \beta \liprodt{\partial_t \vec{m}_{h,k}}{\vec{m}_{h,k} \times \boldsymbol{\varphi}} &- \liprodt{\vec{m}_{h,k} \times \partial_t \vec{m}_{h,k}}{\vec{m}_{h,k} \times \boldsymbol{\varphi}} + \alpha \liprodt{\nabla \vec{m}_{h,k}}{\nabla (\vec{m}_{h,k} \times \boldsymbol{\varphi})} \\
        &+ \liprodt{\vec{m}_{h,k} \times \vec{f}(\vec{m}_{h,k}, \nabla \vec{m}_{h,k})}{\vec{m}_{h,k} \times \boldsymbol{\varphi}} = O(h + k).
    \end{split}
\end{align*}
\end{lemma}
\begin{proof}
Using Lemma \ref{section-4-discrete-weak-formulation} we deduce that
\begin{align*}
    \begin{split}
        \beta \liprodt{\partial_t \vec{m}_{h,k}}{\vec{m}_{h,k} \times \boldsymbol{\varphi}} &- \liprodt{\vec{m}_{h,k} \times \partial_t \vec{m}_{h,k}}{\vec{m}_{h,k} \times \boldsymbol{\varphi}} + \alpha \liprodt{\nabla \vec{m}_{h,k}}{\nabla (\vec{m}_{h,k} \times \boldsymbol{\varphi})} \\
        &+ \liprodt{\vec{m}_{h,k} \times \vec{f}(\vec{m}_{h,k}, \nabla \vec{m}_{h,k})}{\vec{m}_{h,k} \times \boldsymbol{\varphi}}  = I_1 + I_2 + I_3 + I_4 + O(h),
    \end{split}
\end{align*}
where
\begin{align*}
    I_1 &= \beta \liprodt{\partial_t \vec{m}_{h,k}}{\vec{m}_{h,k} \times \boldsymbol{\varphi}} - \beta \liprodt{\vec{v}_{h,k}}{\vec{m}_{h,k}^- \times \boldsymbol{\varphi}}, \\
    I_2 &= - \liprodt{\vec{m}_{h,k} \times \partial_t \vec{m}_{h,k}}{\vec{m}_{h,k} \times \boldsymbol{\varphi}} + \liprodt{\vec{m}_{h,k}^- \times \vec{v}_{h,k}}{\vec{m}_{h,k}^- \times \boldsymbol{\varphi}}, \\
    I_3 &= \alpha \liprodt{\nabla (\vec{m}_{h,k}^- + \theta k \vec{v}_{h,k})}{\nabla (\vec{m}_{h,k}^- \times \boldsymbol{\varphi})} - \alpha \liprodt{\nabla \vec{m}_{h,k}}{\nabla (\vec{m}_{h,k} \times \boldsymbol{\varphi})}, \\
    I_4 &= \liprodt{\vec{m}_{h,k}^- \times \vec{f}(\vec{m}_{h,k}^-, \nabla \vec{m}_{h,k}^-)}{\vec{m}_{h,k}^- \times \boldsymbol{\varphi}} - \liprodt{\vec{m}_{h,k} \times \vec{f}(\vec{m}_{h,k}, \nabla \vec{m}_{h,k})}{\vec{m}_{h,k} \times \boldsymbol{\varphi}}.
\end{align*}
We now proceed to show that $I_i = O(k)$ for $i=1,2,3,4$. The proofs for $I_1, I_2$ and $I_3$ have been given in \cite[Lemma 6.5]{goldys2016-article}. For $I_4$, we first note that
\begin{align*}
    &\Big|\liprodt{\vec{m}_{h,k}^- \times \vec{f}(\vec{m}_{h,k}^-, \nabla \vec{m}_{h,k}^-)}{\vec{m}_{h,k}^- \times \boldsymbol{\varphi}} - \liprodt{\vec{m}_{h,k} \times \vec{f}(\vec{m}_{h,k},\nabla \vec{m}_{h,k})}{\vec{m}_{h,k} \times \boldsymbol{\varphi}}\Big| \\
    &\qquad \lsim \Big|\liprodt{\vec{m}_{h,k}^- \times \vec{f}(\vec{m}_{h,k}^-, \nabla \vec{m}_{h,k}^-)}{\vec{m}_{h,k}^- \times \boldsymbol{\varphi}} - \liprodt{\vec{m}_{h,k} \times \vec{f}(\vec{m}_{h,k}^-, \nabla \vec{m}_{h,k}^-)}{\vec{m}_{h,k} \times \boldsymbol{\varphi}}\Big| \\
    &\qquad \qquad + \Big|\liprodt{\vec{m}_{h,k} \times \vec{f}(\vec{m}_{h,k}^-, \nabla \vec{m}_{h,k}^-)}{\vec{m}_{h,k} \times \boldsymbol{\varphi}} \\
    &\qquad \qquad \qquad- \liprodt{\vec{m}_{h,k} \times \vec{f}(\vec{m}_{h,k}, \nabla \vec{m}_{h,k})}{\vec{m}_{h,k} \times \boldsymbol{\varphi}}\Big|.
\end{align*}
The first expression tends to zero for the same reason that $I_1$ and $I_2$ do, by making use of Lemma \ref{lem:f-linf-bound}. We now direct our attention to the second expression, which is equivalent to analysing the convergence of 
\begin{align*}
    &\Big|\liprodt{\vec{f}(\vec{m}_{h,k}^-, \nabla \vec{m}_{h,k}^-)}{(\vec{m}_{h,k} \times \boldsymbol{\varphi}) \times \vec{m}_{h,k}} - \liprodt{\vec{f}(\vec{m}_{h,k}, \nabla \vec{m}_{h,k})}{(\vec{m}_{h,k} \times \boldsymbol{\varphi}) \times \vec{m}_{h,k}}\Big| \\
    &\qquad \lsim \Big|\liprodt{\vec{f}_1(\vec{m}_{h,k}^-)}{(\vec{m}_{h,k} \times \boldsymbol{\varphi}) \times \vec{m}_{h,k}} - \liprodt{\vec{f}_1(\vec{m}_{h,k})}{(\vec{m}_{h,k} \times \boldsymbol{\varphi}) \times \vec{m}_{h,k}}\Big| \\
    &\qquad \qquad + \Big|\liprodt{\vec{f}_2(\vec{m}_{h,k}^-) \times \vec{g}_1(\nabla \vec{m}_{h,k}^-)}{(\vec{m}_{h,k} \times \boldsymbol{\varphi}) \times \vec{m}_{h,k}} \\
    &\qquad \qquad \qquad - \liprodt{\vec{f}(\vec{m}_{h,k}) \times \vec{g}_1(\nabla \vec{m}_{h,k})}{(\vec{m}_{h,k} \times \boldsymbol{\varphi}) \times \vec{m}_{h,k}}\Big| \\
    &\qquad \qquad + \Big|\liprodt{\vec{f}_3(\vec{m}_{h,k}^-) \times (\vec{f}_4(\vec{m}_{h,k}^-) \times \vec{g}_2(\nabla \vec{m}_{h,k}^-))}{(\vec{m}_{h,k} \times \boldsymbol{\varphi}) \times \vec{m}_{h,k}} \\
    &\qquad \qquad \qquad - \liprodt{\vec{f}_3(\vec{m}_{h,k}) \times (\vec{f}_4(\vec{m}_{h,k}) \times \vec{g}_2(\nabla \vec{m}_{h,k}))}{(\vec{m}_{h,k} \times \boldsymbol{\varphi}) \times \vec{m}_{h,k}}\Big|
\end{align*}
The first component term of $\vec{f}$, i.e., $\vec{f}_1$, is $O(k)$ due to Lemma \ref{discrete-interpolant-error} as $\vec{f}_1$ is $C^1$ by Assumption (A1) and thus is locally Lipschitz by virtue of Lemma \ref{linf-bound}. For the second component we obtain
\begin{align*}
    &\Big|\liprodt{\vec{f}_2(\vec{m}_{h,k}^-) \times \vec{g}_1(\nabla \vec{m}_{h,k}^-)}{(\vec{m}_{h,k} \times \boldsymbol{\varphi}) \times \vec{m}_{h,k}} \\
    &\qquad - \liprodt{\vec{f}(\vec{m}_{h,k}) \times \vec{g}_1(\nabla \vec{m}_{h,k})}{(\vec{m}_{h,k} \times \boldsymbol{\varphi}) \times \vec{m}_{h,k}}\Big| \\
    &\qquad \qquad \lsim \Big|\liprodt{(\vec{f}_2(\vec{m}_{h,k}^-) - \vec{f}_2(\vec{m}_{h,k})) \times \vec{g}_1(\nabla \vec{m}_{h,k}^-)}{(\vec{m}_{h,k} \times \boldsymbol{\varphi}) \times \vec{m}_{h,k}}\Big| \\
    &\qquad \qquad \qquad + \Big|\liprodt{\vec{f}_2(\vec{m}_{h,k}) \times (\vec{g}_1(\nabla \vec{m}_{h,k}^-) - \vec{g}_1(\nabla \vec{m}_{h,k}))}{(\vec{m}_{h,k} \times \boldsymbol{\varphi}) \times \vec{m}_{h,k}}\Big| \\
    &\qquad \qquad \lsim \lnorm{\vec{m}_{h,k}}{\infty}^2 \lnorm{\boldsymbol{\varphi}}{\infty} \lnorm{\vec{g}_1(\nabla \vec{m}_{h,k})}{2} \lnorm{\vec{f}_2(\vec{m}_{h,k}^-) - \vec{f}_2(\vec{m}_{h,k})}{2} \\
    &\qquad \qquad \qquad + \Big|\liprodt{\vec{g}_1(\nabla \vec{m}_{h,k}^-) - \vec{g}_1(\nabla\vec{m}_{h,k})}{((\vec{m}_{h,k} \times \boldsymbol{\varphi}) \times \vec{m}_{h,k}) \times \vec{f}_2(\vec{m}_{h,k})}\Big|.
\end{align*}
The former term in this expression is $O(k)$ by using the Lipschitz property of $\vec{f}_2$ as before. The latter term is also $O(k)$ because $\vec{g}_1$ is linear, and thus can be represented by a matrix. Taking the matrix to the other side of the inner-product and using integration by parts as for $I_3$, we get that this term is $O(k)$ as well. The third component similarly follows suit.
\end{proof}

Lemma \ref{interpolant-bounds}, the Banach-Alaoglu theorem and compactness arguments imply the existence of a sub-sequences $\{\vec{m}_{h,k}\}$ and $\{\vec{v}_{h,k}\}$ that converge in the following senses as $(h,k) \to (0,0)$:
\begin{align}
        \vec{m}_{h,k} \weakto \vec{m} &\qquad \text{in } \HB^1(\Omega_T), \label{section-4-convergence-1} \\[1ex] 
        \vec{m}_{h,k} \to \vec{m} &\qquad \text{in } \LB^2(\Omega_T), \label{section-4-convergence-2} \\[1ex] 
        \vec{v}_{h,k} \weakto \vec{v} &\qquad \text{in } \LB^2(\Omega_T). \label{section-4-convergence-3}
\end{align}
From Lemma \ref{section-4-interpolant-weak-formulation} we know that for all $\boldsymbol{\varphi} \in C_0^{\infty}(\Omega_T)$
\begin{align*}
    \begin{split}
        \beta \liprodt{\partial_t \vec{m}_{h,k}}{\vec{m}_{h,k} \times \boldsymbol{\varphi}} &- \liprodt{\vec{m}_{h,k} \times \partial_t \vec{m}_{h,k}}{\vec{m}_{h,k} \times \boldsymbol{\varphi}} + \alpha \liprodt{\nabla \vec{m}_{h,k} }{\nabla (\vec{m}_{h,k} \times \boldsymbol{\varphi})} \\
        &+ \liprodt{\vec{m}_{h,k} \times \vec{f}(\vec{m}_{h,k}, \nabla \vec{m}_{h,k})}{\vec{m}_{h,k} \times \boldsymbol{\varphi}}  = O(h + k).
    \end{split}
\end{align*}
Following the arguments in \cite[Theorem 6.8]{goldys2016-article} and \cite[Theorem 4.5]{le2013-article}, we can prove that the weak formulation above is satisfied by the limit $\vec{m}$ as $(h,k) \to (0,0)$, except when it comes to the inner-product containing the non-homogeneous term. For this, we use the same strategy as in Lemma \ref{section-4-interpolant-weak-formulation} to show convergence component-wise. For the first component, using the same algebra as before we obtain
\begin{align*}
    &\Big|\liprodt{\vec{f}_1(\vec{m}_{h,k})}{(\vec{m} \times \boldsymbol{\varphi}) \times \vec{m}} - \liprodt{\vec{f}_1(\vec{m})}{(\vec{m} \times \boldsymbol{\varphi}) \times \vec{m}}\Big| \\
    &\qquad \lsim \lnorm{\vec{f}_1(\vec{m}_{h,k}) - \vec{f}_1(\vec{m})}{2} \to 0,
\end{align*}
by the locally Lipschitz property of $\vec{f}_1$. Similarly, 
\begin{align*}
    &\Big|\liprodt{\vec{f}_2(\vec{m}_{h,k}) \times \vec{g}_1(\nabla \vec{m}_{h,k})}{(\vec{m} \times \boldsymbol{\varphi}) \times \vec{m}} - \liprodt{\vec{f}(\vec{m}) \times \vec{g}_1(\nabla \vec{m})}{(\vec{m} \times \boldsymbol{\varphi}) \times \vec{m}}\Big| \\
    &\qquad \lsim \lnorm{\vec{f}_2(\vec{m}_{h,k}) - \vec{f}_2(\vec{m})}{2} \\
    &\qquad \qquad + \Big|\liprodt{\vec{g}_1(\nabla \vec{m}_{h,k}) - \vec{g}_1(\nabla\vec{m})}{((\vec{m} \times \boldsymbol{\varphi}) \times \vec{m}) \times \vec{f}_2(\vec{m})}\Big| \to 0,
\end{align*}
by the locally Lipschitz property of $\vec{f}_2$ and weak convergence. The same can be shown for the third component as well. Hence as $(h,k) \to (0,0)$, we conclude that $\vec{m}$ satisfies 
\begin{align*}
    \begin{split}
        \beta \liprodt{\partial_t \vec{m}}{\vec{m} \times \boldsymbol{\varphi}} - \liprodt{\vec{m} \times \partial_t \vec{m}}{\vec{m} \times \boldsymbol{\varphi}} &+ \alpha \liprodt{\nabla \vec{m} }{\nabla (\vec{m} \times \boldsymbol{\varphi})} \\
        &+ \liprodt{\vec{m} \times \vec{f}(\vec{m}, \nabla \vec{m})}{\vec{m} \times \boldsymbol{\varphi}} = 0,
    \end{split}
\end{align*}
for all $\boldsymbol{\varphi} \in C_0^{\infty}(\Omega_T)$.

Furthermore, $|\vec{m}| = 1$ by (\ref{discrete-interpolant-error-4}) and the initial data is attained by $\vec{m}$ because of the limiting properties of the interpolation operator. This completes the proof.

\section{Application and Numerical Experiments} \label{section-applications}

According to Garello et al. \cite{garello2013-article}, ``Memory and logic spintronic devices rely on the generation of spin torques to control the magnetisation of nanoscale elements using electric currents''. Accounting for these spin-torque effects requires appending a spin-torque term $\vec{T}$ to the Landau-Lifshitz-Gilbert equation
\begin{equation*}
    \frac{\partial \vec{m}}{\partial t} = \alpha\vec{m} \times \Delta \vec{m} - \beta \vec{m} \times \frac{\partial \vec{m}}{\partial t} + \vec{T},
\end{equation*}
in other words, $\vec{f} = \vec{T}$ here.

\subsection{Example 1} \label{sec:example-1}
The spin-transfer torque (STT) method produces spin torques by using a current to transfer spin angular momentum using a `polariser-ferromagnetic layer' and is represented by $\vec{T} = \vec{T}_{\rm STT}$ \cite{li2004-article, meo2022-article, ralph2008-article}, where
\begin{equation*}
    \vec{T}_{\rm STT}\left(\vec{m}, \nabla \vec{m}\right) = \lambda \vec{m} \times (\vec{j} \cdot \nabla \vec{m}) + \mu \vec{m} \times \left(\vec{m} \times (\vec{j} \cdot \nabla \vec{m})\right),
\end{equation*}
and $\vec{j}$ is a constant unit vector denoting the direction of the current. Choosing $\vec{f}_1 = \vec{0}$ but $\vec{f}_2(\vec{a}) = \vec{f}_3(\vec{a}) = \vec{f}_4(\vec{a}) = \vec{a}$ and $\vec{g}_1(\vec{B}) = \vec{g}_2(\vec{B}) = \vec{j} \cdot \vec{B}$, while ignoring the constants for simplicity, it is obvious that these terms satisfy Assumptions (A1) and (A2).

\subsection{Example 2} \label{sec:example-2}
More recent experiments have shown such a `polariser-ferromagnetic layer' is not necessarily required and spin-torques can be induced by exploiting other physical effects. These torques are called `spin-orbit torques' (SOT) and are represented by $\vec{T} = \vec{T}_{\rm SOT}$ \cite{garello2013-article}, where
\begin{equation*}
    \vec{T}_{\rm SOT}(\vec{m}) = \vec{T}^{\perp}(\vec{m}) + \vec{T}^{\parallel}(\vec{m}).
\end{equation*}
Here we have
\begin{align*}
    \vec{T}^{\parallel}(\vec{m}) &= \left[c_1 + c_2 \left|\vec{\hat{k}} \times \vec{m}\right|^2 + c_3 \left|\vec{\hat{k}} \times \vec{m}\right|^4\right] \left(\vec{\hat{\jmath}} \times \vec{m}\right) + \left[c_4 + c_5 \left|\vec{\hat{k}} \times \vec{m}\right|^2\right] \vec{m} \times \left(\vec{\hat{k}} \times \vec{m}\right) \left(\vec{m} \cdot \vec{\hat{\imath}}\right)
\end{align*}
and
\begin{align*}
    \vec{T}^{\perp}(\vec{m}) &= c_6 \vec{m} \times \left(\vec{\hat{\jmath}} \times \vec{m}\right) + \left[c_7 + c_8 \left|\vec{\hat{k}} \times \vec{m}\right|^2\right] \left(\vec{\hat{k}} \times \vec{m}\right) (\vec{m} \cdot \vec{\hat{\imath}}),
\end{align*}
where $c_1, \ldots, c_8$ are experimentally determined physical constants and $\vec{\hat{\imath}}, \vec{\hat{\jmath}}, \vec{\hat{k}}$ are the standard basis vectors in $\R^3$. Choosing $\vec{f}_1(\vec{a}) = \vec{T}(\vec{a})$ and $\vec{f}_2, \vec{f}_3, \vec{f}_4, \vec{g}_1, \vec{g}_2 = \vec{0}$, while ignoring the constants for simplicity, it is clear that this function satisfies Assumption (A1) and the orthogonality property of (A2). We now show that it satisfies the growth conditions of Assumption (A2). First, we obtain
\begin{align*}
    |\vec{T}^{\perp}(\vec{a})| \lsim |\vec{a}|^2 + (1 + |\vec{a}|^2) |\vec{a}|^2 \lsim 1 + |\vec{a}|^4,
\end{align*}
by Young's inequality. Moreover, taking the gradient, we have
\begin{align*}
    |\nabla \vec{T}^{\perp}(\vec{a})| &= \Big| c_6 \nabla \vec{a} \times \left(\vec{\hat{\jmath}} \times \vec{a}\right) + c_6 \vec{a} \times \left(\vec{\hat{\jmath}} \times \nabla \vec{a}\right) + \left[c_7 + 2 c_8 (\vec{\hat{k}} \times \nabla \vec{a}) \cdot (\vec{\hat{k}} \times \vec{a})\right] \left(\vec{\hat{k}} \times \vec{a}\right) (\vec{a} \cdot \vec{\hat{\imath}}) \\
    &\qquad + \left[c_7 + c_8 \left|\vec{\hat{k}} \times \vec{a}\right|^2\right] \left(\vec{\hat{k}} \times \nabla \vec{a}\right) \left(\vec{a} \cdot \vec{\hat{\imath}}\right) + \left[c_7 + c_8 \left|\vec{\hat{k}} \times \vec{a}\right|^2\right] \left(\vec{\hat{k}} \times \vec{a}\right) \left(\nabla \vec{a} \cdot \vec{\hat{\imath}}\right)\Big|.
\end{align*}
As $\nabla \vec{a} = I$, we also obtain
\begin{align*}
    |\nabla \vec{T}^{\perp}(\vec{a})| &\lsim 1 + |\vec{a}| + (1 + |\vec{a}|) |\vec{a}|^2 + (1 + |\vec{a}|^2) |\vec{a}| + (1 + |\vec{a}|^2) |\vec{a}| \lsim 1 + |\vec{a}|^3 \lsim 1 + |\vec{a}|^4.
\end{align*}
In the same way, it can be shown that the parallel term satisfies
\begin{equation*}
    |\vec{T}^{\parallel}(\vec{a})| + |\nabla \vec{T}^{\parallel}(\vec{a})| \lsim 1 + |\vec{a}|^5.
\end{equation*}

\subsection{Numerical Experiments}
In our experiments we choose $\Omega = \left(-\frac{1}{2}, \frac{1}{2}\right) \times \left(-\frac{1}{2}, \frac{1}{2}\right)$ and
\begin{equation*}
    \vec{m}_0(\vec{x}) = 
    \begin{cases}
        (4 A\vec{x}, A^2 - 4|\vec{x}|^2) / (A^2 + 4|\vec{x}|^2), & |\vec{x}| < \frac{1}{4}, \\
        (-4  A\vec{x}, A^2 - 4 |\vec{x}|^2) / (A^2 + 4 |\vec{x}|^2), & \frac{1}{4} \leq |\vec{x}| < \frac{1}{2}, \\
        (-\vec{x}, 0) / |\vec{x}|, & |\vec{x}| \geq \frac{1}{2},
    \end{cases}
\end{equation*}
where $A = (1 - 4|\vec{x}|)^4$. We carry out the experiments for both examples presented in Subsections \ref{sec:example-1} and \ref{sec:example-2}, namely $\vec{T}_{\rm STT}$ and $\vec{T}_{\rm SOT}$. We set all constants to 1 and choose $\vec{j} = (1,0)^T$. We set our $\theta$-scheme parameter to be $\theta = \frac{3}{4}$ with $T = 5$ and $N = 1000$. Images of the simulation are given in Figures \ref{fig:stt-experiment} and \ref{fig:sot-experiment} while a plot of the exchange energy for both cases are given in Figure \ref{fig:energy-plot}. Figures \ref{fig:stt-experiment} and \ref{fig:sot-experiment} show that the magnetisation vectors will eventually line up in one direction, which corresponds to some local minimiser of the total energy (i.e., the free energy) of the system.

\begin{figure}
\centering
\subfloat[$t=0$]{\includegraphics[trim={35cm 5cm 35cm 5cm}, clip=true, scale=0.15]{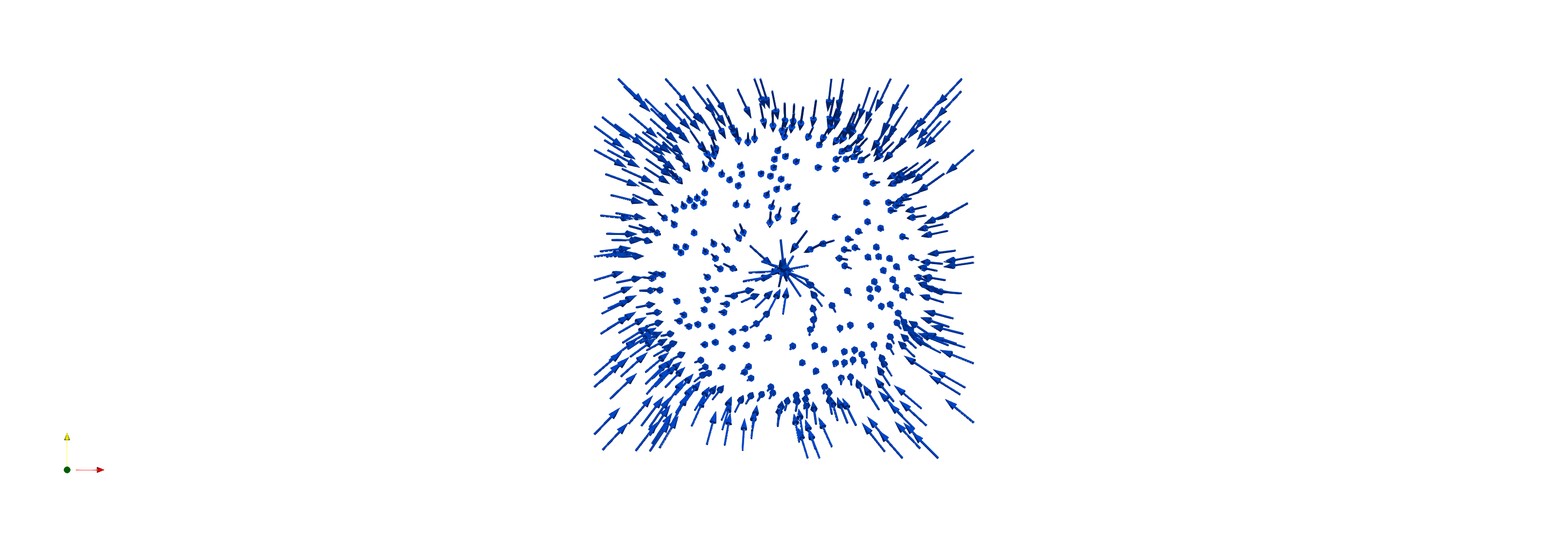}} 
\subfloat[$t=0.3$]{\includegraphics[trim={35cm 5cm 35cm 5cm}, clip=true, scale=0.15]{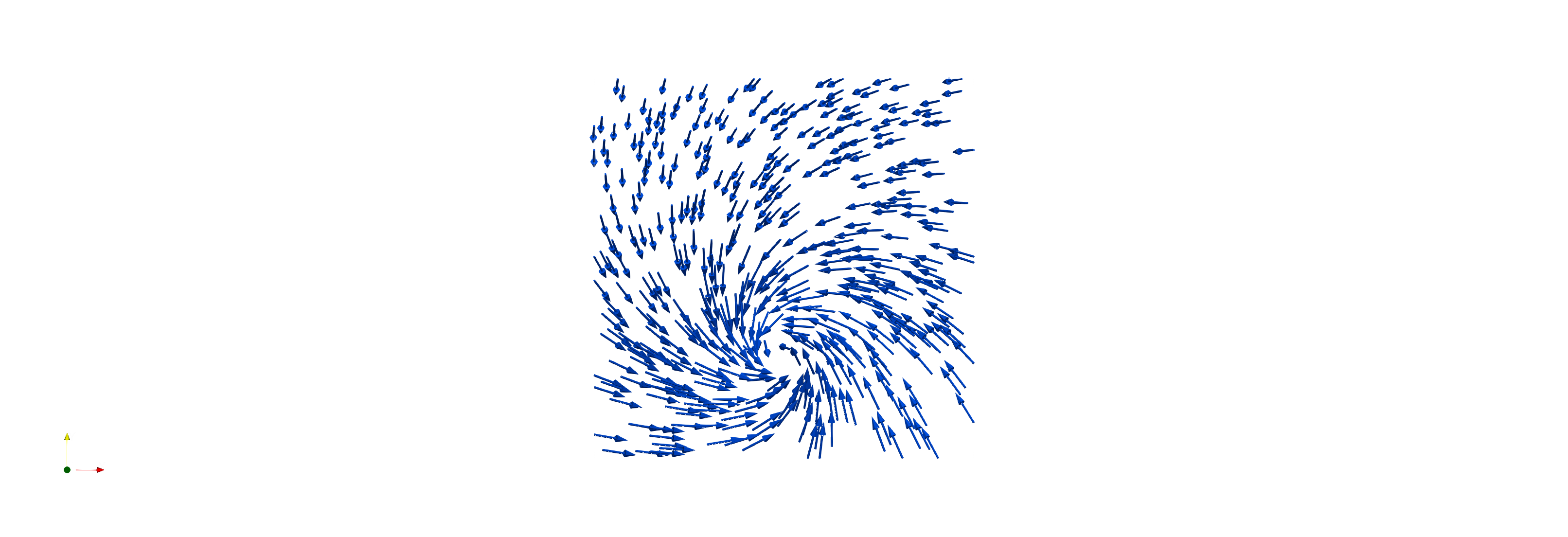}}\\
\subfloat[$t=0.6$]{\includegraphics[trim={35cm 5cm 35cm 5cm}, clip=true, scale=0.15]{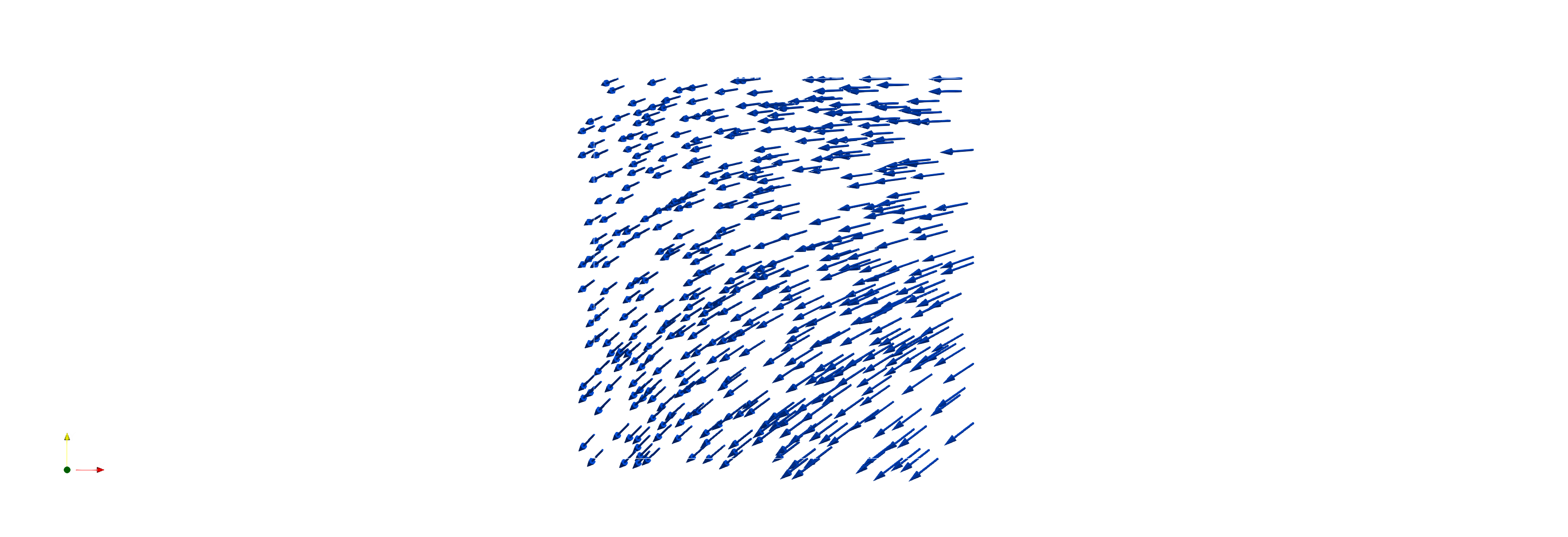}}
\subfloat[$t=0.9$]{\includegraphics[trim={35cm 5cm 35cm 5cm}, clip=true, scale=0.15]{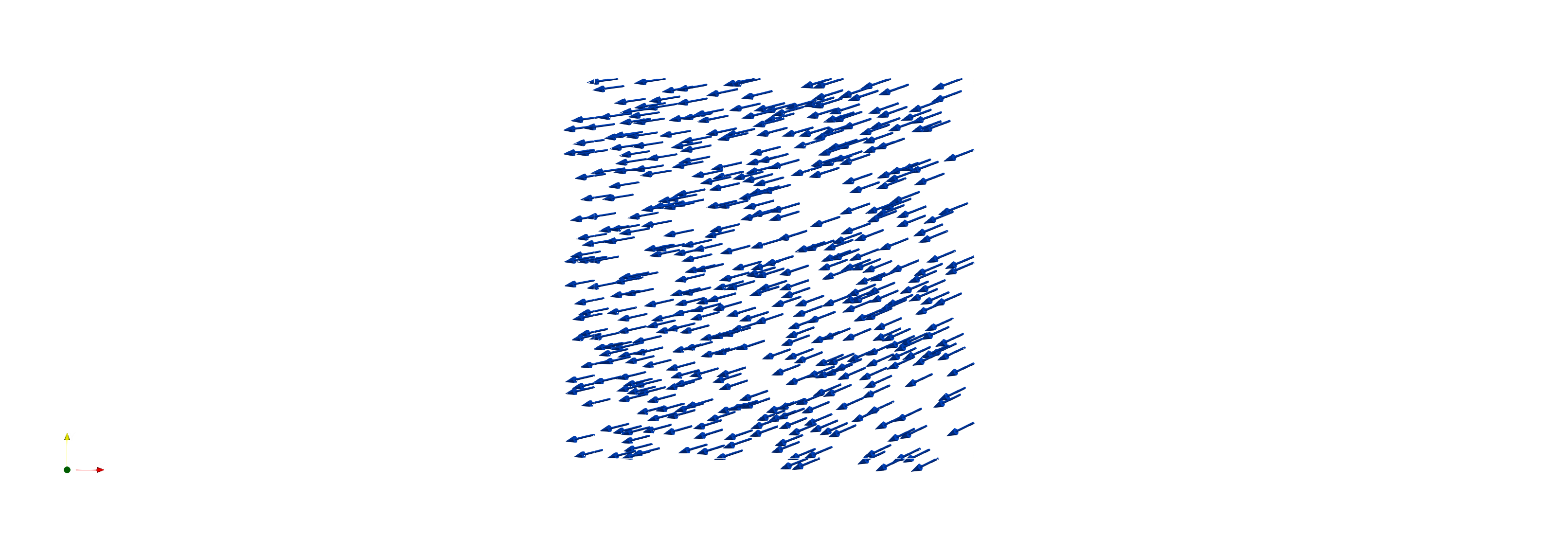}} 
\caption{Spin Transfer Torque}
\label{fig:stt-experiment}
\end{figure}

\begin{figure}
\centering
\subfloat[$t=0$]{\includegraphics[trim={35cm 5cm 35cm 5cm}, clip=true, scale=0.15]{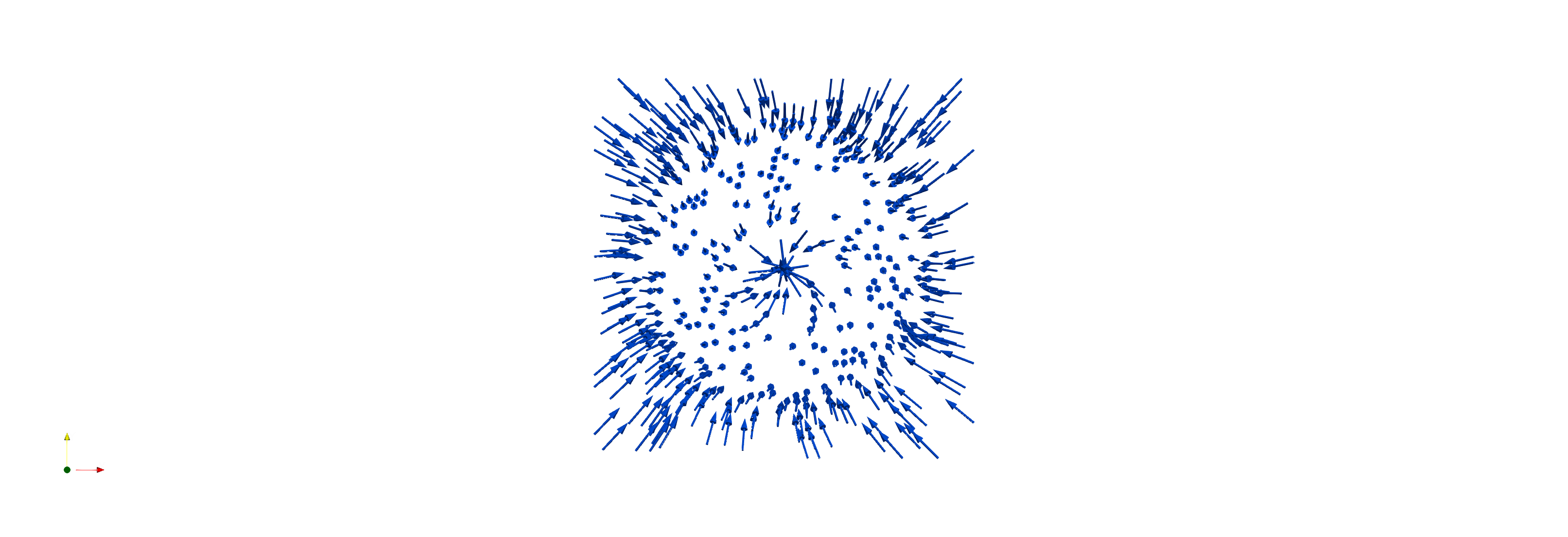}} 
\subfloat[$t=0.6$]{\includegraphics[trim={35cm 5cm 35cm 5cm}, clip=true, scale=0.15]{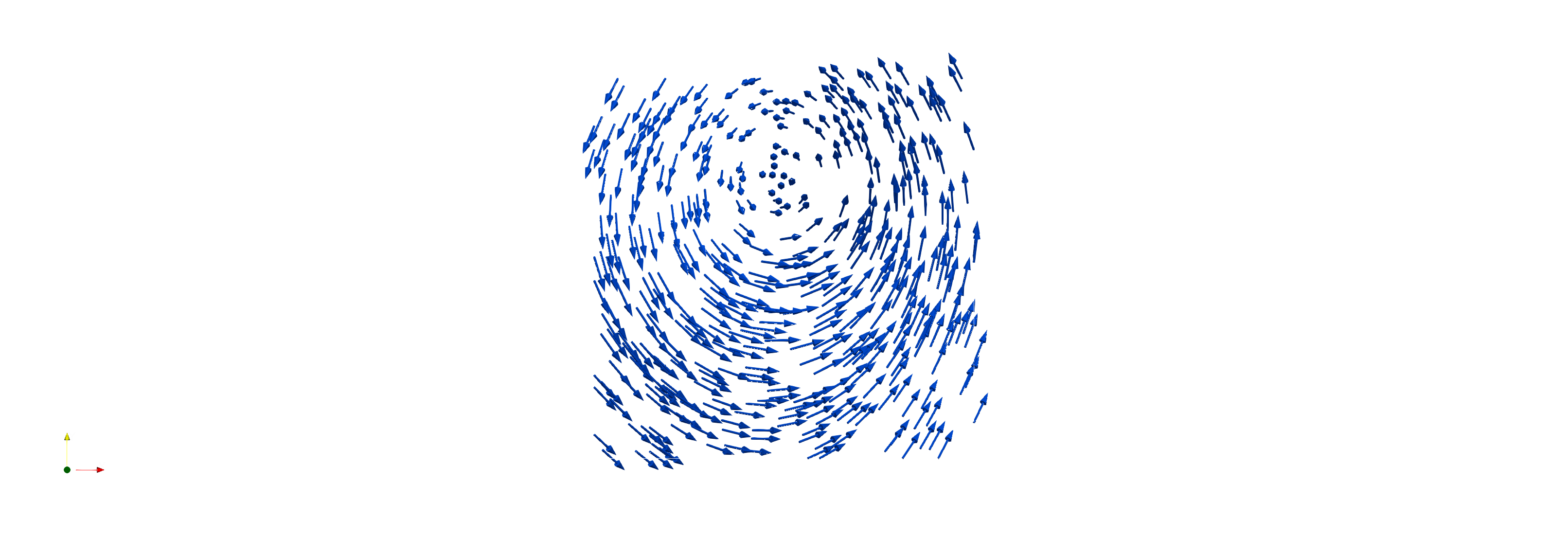}}\\
\subfloat[$t=1.2$]{\includegraphics[trim={35cm 5cm 35cm 5cm}, clip=true, scale=0.15]{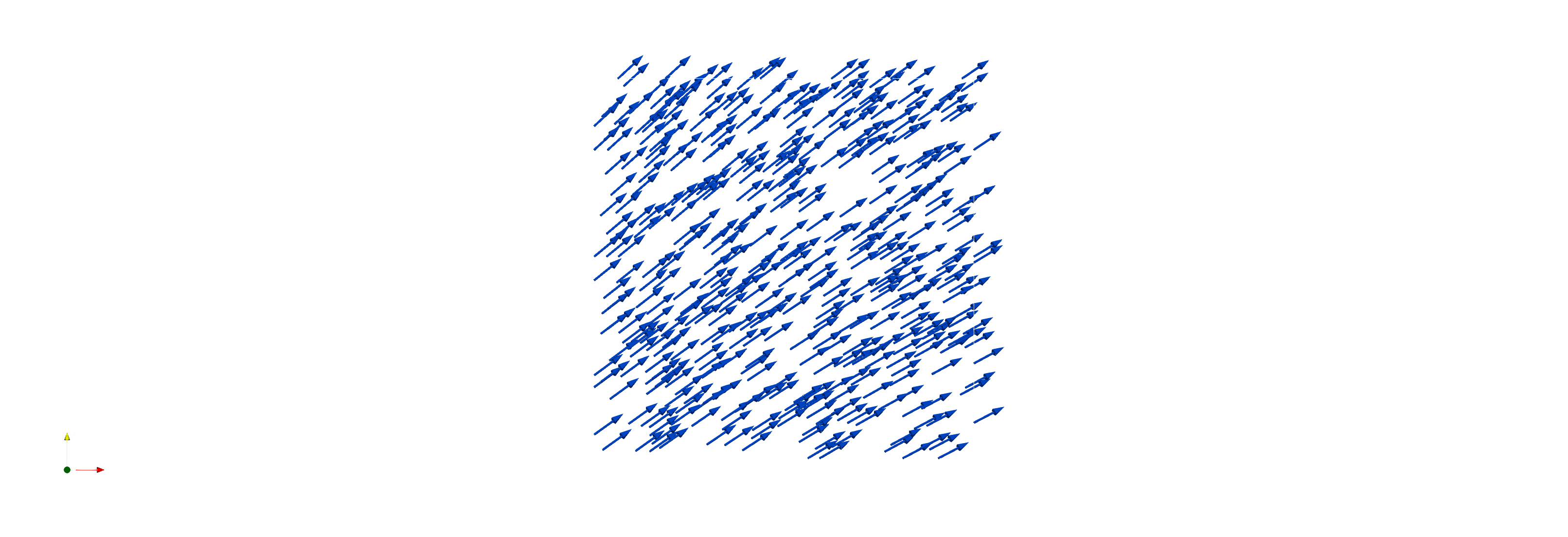}}
\subfloat[$t=1.8$]{\includegraphics[trim={35cm 5cm 35cm 5cm}, clip=true, scale=0.15]{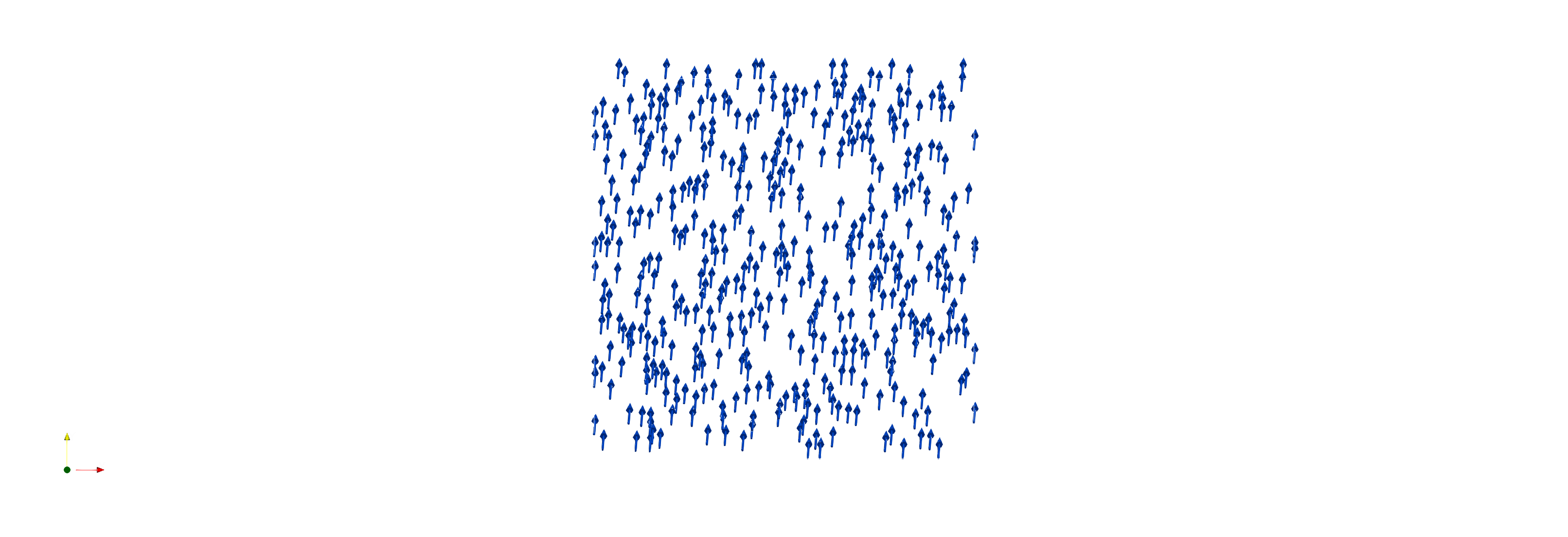}} 
\caption{Spin Orbit Torque}
\label{fig:sot-experiment}
\end{figure}

\begin{figure}
    \centering
    \includegraphics[scale=0.5]{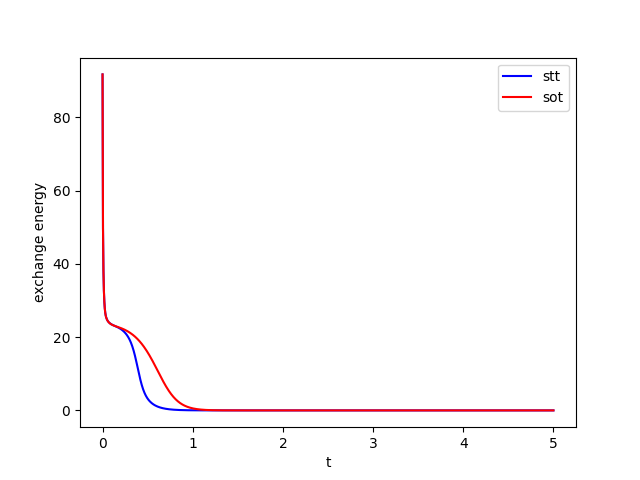}
    \caption{Plot of the energy, $t \mapsto \lnorm{\nabla \vec{m}_{h,k}(t)}{2}^2$}
    \label{fig:energy-plot}
\end{figure}

\section*{Acknowledgements}
The first author is supported by the Australian Government's Research Training Program Scholarship awarded at the University of New South Wales, Sydney. The second author is partially supported by the Australian Research Council under grant number DP190101197 and DP200101866.

\appendix
\section{Useful Results} \label{section-appendix-a}

\begin{theorem}[Aubin's Lemma \cite{aubin1963-article, simon1986-article}] \label{aubins-lemma}
Let $X$, $Y$ and $B$ be Banach spaces such that $X \subset B \subset Y$, where the injection $X \subset B$ is compact and the injection $B \subset Y$ is continuous. Assume that $\{u_k\}_{k=1}^{\infty}$ is a bounded sequence in $L^p(0,T; X)$ such that $\{\partial_t u_k\}_{k=1}^{\infty}$ is bounded in $L^r(0,T; Y)$ where $1 \leq p < \infty$ and $r = 1$, or $p = \infty$ and $r > 1$. Then there exists a subsequence $\{u_{k_j}\}_{j=1}^{\infty}$ which strongly converges in $L^p(0,T; B)$. 
\end{theorem}

\begin{theorem}[Generalised Gronwall Lemma \cite{beckenbach1961-book}] \label{generalised-gronwall-lemma}
Assume $u : [a,b] \to [0,\infty)$ and $\beta : [a,b] \to [0,\infty)$ along with non-decreasing $g : [0,\infty) \to [0,\infty)$ satisfy
\begin{equation*}
    u(t) \leq \alpha + \int_a^t \beta(s) g(u(s)) \; {\rm d}s \quad \forall t \in [a,b],
\end{equation*}
where $\alpha$ is a positive constant and $[a,b] \subset [0,\infty)$. Then
\begin{equation*}
    u(t) \leq G^{-1}\left(\int_a^t \beta(s) \; {\rm d}s\right), \quad t \in \RR,
\end{equation*}
where $G^{-1}$ is the inverse function of
\begin{equation*}
    G(\sigma) := \int_{\alpha}^{\sigma} \frac{1}{g(s)} \; {\rm d}s, \quad \sigma \geq 0,
\end{equation*}
and
\begin{equation*}
    \RR = \left\{t \in [a,b] : \int_a^t \beta(s) \; {\rm d}s \in G([0,\infty))\right\}.
\end{equation*}
\end{theorem}

\begin{theorem}[\cite{evans2010-book}] \label{time-continuity-theorem}
Assume that $\Omega$ is open, bounded, and $\partial \Omega$ is smooth. Take $m$ to be a non-negative integer. Suppose that $\vec{u} \in L^2(0,T; \HB^{m+2}(\Omega))$, with $\partial_t \vec{u} \in L^2(0,T; \HB^m(\Omega))$. Then
\begin{equation*}
    \vec{u} \in C([0,T]; \HB^{m+1}(\Omega)).
\end{equation*}
\end{theorem}

\begin{theorem}[\cite{bartels2006vol43-article}] \label{section-2-bartels-theorem}
    For piecewise linear finite elements, assume that for all $1 \leq i,j \leq J$ such that $i \neq j$
    \begin{equation} \label{section-2-bartels-condition}
        \int_{\Omega} \nabla \phi_i \cdot \nabla \phi_j \; {\rm d}\vec{x}\leq 0.
    \end{equation}
    Then for all $\vec{u} \in \V_h$ satisfying $|\vec{u}(\vec{x}_n)| \geq 1$, where $1 \leq n \leq N$, we have
    \begin{equation*}
        \int_{\Omega} \left|\nabla I_{\V_h}\left(\frac{\vec{u}}{|\vec{u}|}\right)\right|^2 \; {\rm d}\vec{x} \leq \int_{\Omega} |\nabla \vec{u}|^2 \; {\rm d} \vec{x}.
    \end{equation*}
\end{theorem}
It can be shown that the assumptions of the theorem above are satisfied in 2D for Delaunay triangulations and in 3D for triangulations that have their diheral angles less than $\pi/2$ \cite{vanselow2001-article}.

\bibliographystyle{plain}
\bibliography{Bibliography}

\begin{thebibliography}{10}

\bibitem{abert2015-article}
Claas Abert, Michele Ruggeri, Florian Bruckner, Christoph Vogler, Gino Hrkac,
  Dirk Praetorius, and Dieter Seuss.
\newblock {A three-dimensional spin-diffusion model for micromagnetics}.
\newblock {\em Scientific Reports}, 5, 2015.

\bibitem{ado2017-article}
Ivan Ado, Oleg Tretiakov, and Mikhail Titov.
\newblock Microscopic theory of spin-orbit torques in two dimensions.
\newblock {\em Phys. Rev. B}, 95:094401, Mar 2017.

\bibitem{alouges2008-article}
François Alouges.
\newblock A new finite element scheme for {L}andau-{L}ifchitz equations.
\newblock {\em Discrete and Continuous Dynamical Systems - Series S},
  1:187--196, 2008.

\bibitem{alouges1992-article}
François Alouges and Alain Soyeur.
\newblock {On global weak solutions for {L}andau-{L}ifshitz equations:
  Existence and nonuniqueness}.
\newblock {\em Nonlinear Analysis: Theory, Methods \& Applications},
  18(11):1071--1084, 1992.

\bibitem{aubin1963-article}
Jean Aubin.
\newblock {Un théorème de compacité}.
\newblock {\em C. R. Acad. Sci. Paris}, 256:5042–5044, 1963.

\bibitem{ayouch2021-article}
Chahid Ayouch, Kottakkaran Nisar, Mouhcine Tilioua, and M.~Zakarya.
\newblock On the {L}andau-{L}ifshitz-{B}loch equation with spin torque effects.
\newblock {\em Alexandria Engineering Journal}, 60(5):4433--4439, 2021.

\bibitem{bartels2006vol43-article}
Sören Bartels.
\newblock Stability and convergence of finite-element approximation schemes for
  harmonic maps.
\newblock {\em SIAM Journal on Numerical Analysis}, 43(1):220--238, 2006.

\bibitem{bartels2008-article}
Sören Bartels, Joy Ko, and Andreas Prohl.
\newblock Numerical analysis of an explicit approximation scheme for the
  {L}andau-{L}ifshitz-{G}ilbert equation.
\newblock {\em Mathematics of Computation}, 77(262):773--788, 2008.

\bibitem{bartels2006vol44-article}
Sören Bartels and Andreas Prohl.
\newblock Convergence of an implicit finite element method for the
  {L}andau-{L}ifshitz-{G}ilbert equation.
\newblock {\em SIAM Journal on Numerical Analysis}, 44(4):1405--1419, 2006.

\bibitem{banas2008-article}
L'ubom\'{\i}r Ba\v{n}as, S\"{o}ren Bartels, and Andreas Prohl.
\newblock A convergent implicit finite element discretization of the
  {M}axwell-{L}andau-{L}ifshitz-{G}ilbert equation.
\newblock {\em SIAM J. Numer. Anal.}, 46(3):1399--1422, 2008.

\bibitem{beckenbach1961-book}
Edwin Beckenbach and Richard Bellman.
\newblock {\em {Inequalities}}.
\newblock Springer Berlin, Heidelberg, 1st edition, 1961.

\bibitem{brezis2010-book}
Haim Brezis.
\newblock {\em {Function Analysis, Sobolev Spaces and Partial Differential
  Equations}}.
\newblock Springer New York, NY, 01 2010.

\bibitem{carbou2001-article}
Gilles Carbou and Pierre Fabrie.
\newblock {Regular solutions for {L}andau-{L}ifschitz equation in a bounded
  domain}.
\newblock {\em Differential and Integral Equations}, 14(2):213 -- 229, 2001.

\bibitem{chen2000-article}
Yunmei Chen.
\newblock {Dirichlet boundary value problems of {L}andau-{L}ifshitz equation}.
\newblock {\em Communications in Partial Differential Equations},
  25(1-2):101--124, 2000.

\bibitem{chen1998-article}
Yunmei Chen, Shijin Ding, and Boling Guo.
\newblock {Partial regularity for two dimensional {L}andau-{L}ifshitz
  equations}.
\newblock {\em Acta Mathematica Sinica}, 14:423--432, 1998.

\bibitem{evans2010-book}
Lawrence Evans.
\newblock {\em {Partial differential equations}}.
\newblock American Mathematical Society, Providence, R.I., 2010.

\bibitem{garello2013-article}
Kevin Garello, Ioan~M. Miron, Can~O. Avci, Frank Freimuth, Yuriy Mokrousov,
  Stefan Blügel, Stéphane Auffret, Olivier Boulle, Gilles Gaudin, and Pietro
  Gambardella.
\newblock Symmetry and magnitude of spin-orbit torques in ferromagnetic
  heterostructures.
\newblock {\em Nature Nanotechnology}, 8(8):587--93, Aug 2013.

\bibitem{gilbert1955-article}
Thomas Gilbert.
\newblock A {L}agrangian formulation of the gyromagnetic equation of the
  magnetization field.
\newblock {\em Physical Review D}, 100:1243, 1955.

\bibitem{goldys2016-article}
Benjamin Goldys, Kim Le, and Thanh Tran.
\newblock {A finite element approximation for the stochastic
  {L}andau–{L}ifshitz–{G}ilbert equation}.
\newblock {\em Journal of Differential Equations}, 260(2):937--970, 2016.

\bibitem{harpes2004-article}
Paul Harpes.
\newblock {Uniqueness and bubbling of the 2-dimensional {L}andau-{L}ifshitz
  flow}.
\newblock {\em Calculus of Variations}, 20:213--229, 01 2004.

\bibitem{lakshmanan2011-article}
Muthusamy Lakshmanan.
\newblock {The fascinating world of the {L}andau-{L}ifshitz-{G}ilbert equation:
  an overview}.
\newblock {\em Philosophical Transactions: Mathematical, Physical and
  Engineering Sciences}, 369(1939):1280--1300, 2011.

\bibitem{landau1935-article}
Lev Landau and Evgeny Lifshitz.
\newblock {{On the theory of the dispersion of magnetic permeability in
  ferromagnetic bodies}}.
\newblock {\em Phys. Z. Sowjet.}, 8:153, 1935.

\bibitem{le2013-article}
Kim-Ngan Le and Thanh Tran.
\newblock A convergent finite element approximation for the quasi-static
  {M}axwell–{L}andau–{L}ifshitz–{G}ilbert equations.
\newblock {\em Computers \& Mathematics with Applications}, 66(8):1389--1402,
  2013.

\bibitem{li2004-article}
Zhanjie Li and Shuqi Zhang.
\newblock Domain-wall dynamics and spin-wave excitations with spin-transfer
  torques.
\newblock {\em Phys. Rev. Lett.}, 92:207203, May 2004.

\bibitem{melcher2005-article}
Christof Melcher.
\newblock {Existence of partially regular solutions for {L}andau–{L}ifshitz
  equations in $\mathbb{R}^3$}.
\newblock {\em Communications in Partial Differential Equations},
  30(4):567--587, 2005.

\bibitem{melcher2013-article}
Christof Melcher and Mariya Ptashnyk.
\newblock Landau-{L}ifshitz-{S}lonczewski equations: global weak and classical
  solutions.
\newblock {\em SIAM J. Math. Anal.}, 45(1):407--429, 2013.

\bibitem{meo2022-article}
Andrea Meo, Carenza Cronshaw, Sarah Jenkins, Amelia Lees, and Richard Evans.
\newblock Spin-transfer and spin-orbit torques in the
  {L}andau–{L}ifshitz–{G}ilbert equation.
\newblock {\em Journal of Physics: Condensed Matter}, 35(2):025801, nov 2022.

\bibitem{ralph2008-article}
Dan Ralph and Mark Stiles.
\newblock Spin transfer torques.
\newblock {\em Journal of Magnetism and Magnetic Materials}, 320(7):1190--1216,
  2008.

\bibitem{simon1986-article}
Jacques Simon.
\newblock {Compact sets in the space $L^p(0,T;B)$}.
\newblock {\em Annali di Matematica pura ed applicata}, 146:65–96, 1986.

\bibitem{soenjaya2023-article}
Agus Soenjaya and Thanh Tran.
\newblock Global solutions of the {L}andau-{L}ifshitz-{B}aryakhtar equation.
\newblock {\em J. Differential Equations}, 371:191--230, 2023.

\bibitem{sugiyama1969-article}
Shohei Sugiyama.
\newblock Stability problems on difference and functional-differential
  equations.
\newblock {\em Proceedings of the Japan Academy}, 45(7):526--529, 1969.

\bibitem{temam2001-book}
Roger Temam.
\newblock {\em Navier-{S}tokes equations}.
\newblock AMS Chelsea Publishing, Providence, RI, 2001.
\newblock Theory and numerical analysis, Reprint of the 1984 edition.

\bibitem{vanselow2001-article}
Reiner Vanselow.
\newblock About {D}elaunay triangulations and discrete maximum principles for
  the linear conforming fem applied to the {P}oisson equation.
\newblock {\em Applications of Mathematics}, 46:13--28, 2001.

\bibitem{visintin1985-article}
Augusto Visintin.
\newblock {On {L}andau-{L}ifshitz’ equations for ferromagnetism}.
\newblock {\em Japan Journal of Applied Mathematics}, 2:69--84, 1985.

\bibitem{zhou1981-article}
Yu-Lin Zhou and Boling Guo.
\newblock {Existence of weak solution for boundary problems of systems of
  ferro-magnetic chain}.
\newblock {\em Science in China, Ser. A 27}, pages 779--811, 1981.

\end{thebibliography}

\end{document}